\documentclass[11pt]{amsart}
\usepackage[margin=1.5in]{geometry}
\usepackage[foot]{amsaddr}
\usepackage{tikz-cd} 
\usepackage{graphicx}

\usepackage{cancel}
\usepackage{amssymb}

\usepackage{lineno}

\usepackage{comment}

\usepackage[centertags]{amsmath}
\usepackage{amsfonts}
\usepackage{amssymb}
\usepackage{bm}
\usepackage{amsthm}
\usepackage[normalem]{ulem}
\usepackage{dsfont}

\usepackage{cancel}
\usepackage{enumitem}
\usepackage{verbatim}
\usepackage{bbm} 

\usepackage{xcolor}

\usepackage[colorlinks]{hyperref}

\usepackage{cleveref}

\newtheorem{thm}{Theorem}[section]
\newtheorem{bigthm}{Theorem}
 
\newtheorem{prop}[thm]{Proposition}
\newtheorem{defn}[thm]{Definition}
\newtheorem{lem}[thm]{Lemma}
\newtheorem{oss}[thm]{Remark}
\newtheorem{ex}[thm]{Example}

\numberwithin{equation}{section} 

\newcommand{\dd}{\mathrm{d}}
\newcommand\cA{{\mathcal A}}
\newcommand\cB{{\mathcal B}}
\newcommand\cC{{\mathcal C}}
\newcommand\cD{{\mathcal D}}
\newcommand\cE{{\mathcal E}}
\newcommand\cF{{\mathcal F}}

\newcommand\cH{{\mathcal H}}

\newcommand\cJ{{\mathcal J}}
\newcommand\cL{{\mathcal L}}

\newcommand\cN{{\mathcal N}}
\newcommand\cT{{\mathcal T}}
\newcommand\cU{{\mathcal U}}

\newcommand\bB{{\mathbb B}}
\newcommand\bC{{\mathbb C}}

\newcommand\bE{{\mathbb E}}

\newcommand\bN{{\mathbb N}}

\newcommand\bP{{\mathbb P}}

\newcommand\bR{{\mathbb R}}
\newcommand\bT{{\mathbb T}}

\newcommand\bV{{\mathbb V}}
\newcommand\bZ{{\mathbb Z}}

\newcommand\ve{\varepsilon}

\newcommand\Id{\text{Id}}

\newcommand\frh{{\mathfrak h}}

\newcommand{\frm}{\mathfrak{m}}

\newcommand\fC{{\mathbf{C}}}
\newcommand\fn{{\mathfrak{n}}}

\newcommand{\supp}{\operatorname{supp}}
\newcommand \re {{%
\operatorname{Re}
}}
\newcommand \im {{%
\operatorname{Im}
}}

\newcommand{\marginnoter}[1]
           {\mbox{}\marginpar{\tiny\raggedright\hspace{0pt}{{\bf Roberto}$\blacktriangleright$  {\color{blue} #1}}}}

\makeatletter
\def\@tocline#1#2#3#4#5#6#7{\relax
  \ifnum #1>\c@tocdepth 
  \else
    \par \addpenalty\@secpenalty\addvspace{#2}%
    \begingroup \hyphenpenalty\@M
    \@ifempty{#4}{%
      \@tempdima\csname r@tocindent\number#1\endcsname\relax
    }{%
      \@tempdima#4\relax
    }%
    \parindent\z@ \leftskip#3\relax \advance\leftskip\@tempdima\relax
    \rightskip\@pnumwidth plus4em \parfillskip-\@pnumwidth
    #5\leavevmode\hskip-\@tempdima
      \ifcase #1
       \or\or \hskip 1em \or \hskip 2em \else \hskip 3em \fi%
      #6\nobreak\relax
    \hfill\hbox to\@pnumwidth{\@tocpagenum{#7}}\par
    \nobreak
    \endgroup
  \fi}
\makeatother

\begin{document}
\title[Limit theorems for 2d-partially hyperbolic systems]{Limit theorems for toral partially hyperbolic endomorphisms}
\author{Roberto Castorrini \textsuperscript{(1)}}
\address{{\tiny (1)} Scuola Normale Superiore, Piazza dei Cavalieri, 7, Pisa, 56126, Italy.}
\email{roberto.castorrini@gmail.com}
\author{Kasun Fernando \textsuperscript{(2)}}
\address{{\tiny (2)} Brunell University of London, Kingston Ln, Uxbridge UB8 3PH, UK.}
\email{{\tt kasun.fernando@brunel.ac.uk }}
\begin{abstract}
Under natural assumptions on the observable, we prove a Central Limit Theorem, a Berry-Esseen Theorem, and a quantitative Local Limit Theorem for a broad class of partially hyperbolic endomorphisms of the two-dimensional torus. Our results apply, but are not limited to, skew-products and their perturbations, and they remain valid even when the system admits multiple, though finitely many, absolutely continuous ergodic invariant measures.
\vspace{-10pt}
\end{abstract}
\date{\today}

\thanks{RC’s research was partially supported by the research project ‘Dynamics and Information Research Institute—Quantum Information, Quantum Technologies’ within the agreement between UniCredit Bank and Scuola Normale Superiore and by the research project PRIN 2017S35EHN of the MIUR. The authors acknowledge the UMI Group “DinAmicI” (\url{www.dinamici.org}) and the INdAM group GNFM. 
} 
\maketitle

\section{Introduction}\label{sec:intro}
Partially hyperbolic systems have emerged as a central object of interest in dynamical systems, yet they continue to pose significant analytical challenges. Their complexity lies in the presence of neutral Lyapunov exponents, which distinguish them from uniformly hyperbolic systems and make their behavior more intricate. Although the literature on ergodic properties of such systems is extensive (see, for example, \cite{ABV, BV, BuWilk}), many of the results are abstract and qualitative in nature, offering limited guidance for the analysis of concrete examples.

Skew-products (including skew-product fast-slow systems) are an important subclass of partially hyperbolic systems. These have been extensively studied. We refer the readers to \cite{Gou1, BuEs, KKM, SM} and references therein for more details. For systems that are not skew-products, from a quantitative standpoint, results are far less common. Notable exceptions are the series of work by De Simoi and Liverani, particularly \cite{DeLi2}, which rigorously analyses fast-slow systems in two dimensions, models of relevance in physics (cf. \cite{Liv}), and the work by Dolgopyat \cite{Dol1} where systems with an \textit{almost} Markov decomposition are considered. These studies contribute significantly to the theory of limit laws and quantitative ergodic properties, yet the technical machinery they rely on is often quite involved and does not easily generalise to broader classes of systems. 

A more versatile approach has emerged in recent years based on the spectral analysis of transfer operators. This method, central to works such as \cite{AGT} and \cite{F11}, enables a finer study of statistical properties in two-dimensional skew-products, canonical examples of partially hyperbolic maps. Building on this, \cite{CaLi} adapts techniques from uniformly hyperbolic dynamics to establish quasi-compactness for a more general class of partially hyperbolic systems, including fast-slow systems. Their results yield quantitative insights into ergodicity, mixing, and decay of correlations of such systems.

It is well known that the spectral gap of the transfer operator yields key statistical features, such as the existence of finitely many absolutely continuous invariant probabilities (acips), exponential decay of correlations, and, in suitable contexts, limit theorems such as the Central Limit Theorem (CLT) or Local Limit Theorem (LLT) (see \cite{BaGo}). However, when departing from uniformly hyperbolic, invertible, {or Markovian} settings, much less is known. In particular, for partially hyperbolic endomorphisms (noninvertible maps with both expanding and neutral directions) CLT results are scarce and typically confined to the special cases mentioned above.

The present paper contributes to this landscape by establishing the CLT, Barry-Esseen Theorem, and LLTs for a large class of partially hyperbolic endomorphisms on the two-dimensional torus. {In fact, we do the non-trivial task of verifying the conditions (appropriately adapted to our setting) stated in the functional analytic framework for the non-ergodic case, introduced in \cite{HeHe}}. 
So, our results apply not only when there is a unique absolutely continuous invariant measure, associated with a simple peripheral eigenvalue, but also in the more general case where multiple ergodic acips coexist.\\ 

\noindent\textbf{Organization of the paper}. In Section \ref{sec:results}, we introduce the systems under investigation and state the two main results: Theorem \ref{thmA} is the CLT with a rate of convergence result (Barry-Esseen Theorem); Theorem \ref{thmB} is an LLT with a quantitative estimates. Both theorems are proved by applying the abstract framework presented in Appendix \ref{app:AbstractCLT} to the two principal analytical tools of the paper: the transfer operator associated with the map, and its complex perturbation, known as the \textit{twisted transfer operator}. Accordingly, Sections \ref{sec:TraOp} and \ref{sec:twisted} are dedicated to studying the spectral properties of these operators respectively, and in particular, to proving that they are quasi-compact when acting on a space of functions contained within the standard Sobolev space $H^s(\bT^2)$. In Section \ref{sec:proof-thmA}, we combine the results obtained in the previous two sections to prove the main theorems. The paper concludes with an appendix that gathers some known, and some less known, technical results used throughout the text.

\section{The system and the results}\label{sec:results}

\subsection{Assumptions on the map} In this paper we consider the class of systems $F\in \cC^r(\bT^2,\bT^2)$ introduced in \cite{CaLi} and called SVPH ({\it strongly dominated vertical partially hyperbolic systems}).
A set of sufficient conditions for checking whether a system is an SVPH are provided in \cite[Lemma 2.2]{CaLi} and we refer to \cite[Section 2.2]{CaLi} for a detailed discussion, including examples, of the following assumptions.

Let $r \ge 4$ and $F:\mathbb{T}^2\mapsto \mathbb{T}^2$ be a surjective $\mathcal{C}^{r}$ local diffeomorphism of the form 
\begin{equation}
\label{the map F}
F(x,\theta)=(f(x,\theta),\theta +\omega(x,\theta)) \mod \bZ^2.
\end{equation}
We call $F$ an SVPH if $F$ satisfies the following properties \textbf{(A1)}--\textbf{(A5)} given below.\vspace{5pt}

\textbf{(A1)} For all $p\in\bT^2$ we have $\det(D_pF)>0$.
\vspace{5pt}

\textbf{(A2)} There exist $\chi_u\in (0,1)$, $\chi_c\in (0,1]$ and $ 0<\lambda^{-}_c<  1<  \lambda_c^+ < \lambda\le\Lambda$ such that, setting
\begin{equation}\label{eq:cone-def}
\begin{split}
&\mathbf{C}_u:= \lbrace (\xi, \eta)\in \cT_z\mathbb{T}^2  :  |\eta|\le \chi_u|\xi | \rbrace \\
&\mathbf{C}_c:= \lbrace (\xi, \eta)\in  \cT_z\mathbb{T}^2 :  |\xi|\le \chi_c |\eta | \rbrace,
\end{split}
\end{equation}
defining 
\begin{equation}\label{def of lamb^+ mu^+}
\begin{split}
&\lambda^-_n(z):= \inf_{v\in \bR^2\setminus \fC_c }\mathfrak{m}_{F^{n}}(z, v) \qquad \qquad\qquad\; \lambda^+_n(z):= \sup_{v\in \bR^2\setminus \fC_c }\mathfrak{m}_{F^{n}}(z, v),\\
&\lambda_{c,n}^-(z):= \inf_{v\in \fC_c\setminus{\lbrace 0 \rbrace}}\mathfrak{m}^*_{F^{n}}(z, v) \quad\quad\qquad\qquad\lambda_{c,n}^+(z)
:= \sup_{v\in \fC_c\setminus{\lbrace 0 \rbrace}}\mathfrak{m}^*_{F^{n}}(z, v),
\end{split}
\end{equation}
where $ \mathfrak{m}_F(z, v)=\frac{\|D_zF v\|}{\|v\|}$ and $\mathfrak{m}^*_F(z, v)=\frac{\|(D_zF)^{-1} v\|}{\|v\|}$
and letting $\lambda^-_n=\inf_z\lambda^-_n(z)$ and $\lambda^+_n=\sup_z\lambda^+_n(z)$ we assume the following:\\
There exist constants $C_\star\geq 1$ and $\iota_\star \in (0,1)$ such that, for all $z\in\bT^2$ and $n\in\bN\,,$\footnote{$A\Subset B$ means $\overline{A}\subset \mbox{int}(B)\cup \lbrace 0 \rbrace$.}
\begin{equation}
  \begin{aligned}
  D_zF\mathbf{C}_u &\subset \{(\xi,\eta): |\eta|\le \iota_\star\chi_u|\xi|\}\Subset \fC_u; \\ (D_zF)^{-1}\mathbf{C}_c &\subset \{(\xi,\eta): |\xi|\le \iota_\star\chi_c|\eta|\}\,, \label{invariance of cone}
  \end{aligned}
\end{equation}
\begin{equation}
  \begin{aligned}
  C_\star^{-1}({\lambda_c^-})^n&\leq \lambda_{c,n}^-(z)\leq \lambda_{c,n}^+(z)\leq C_\star({\lambda_{c}^+})^n \;;\\ C_\star^{-1}\lambda^n&\leq\lambda^-_n(z)\leq  \lambda^+_n(z)\leq C_\star \Lambda^n \,.\label{partial hyperbolicity 2}
  \end{aligned}
\end{equation}
\vspace{5pt}

\textbf{(A3)} Let $\Upsilon$ be the family of closed curve $\gamma \in \cC^r(\bT, \bT^2)$ such that~\footnote{ As usual we consider equivalent two curves that differ only by a $\cC^r$ reparametrization. In the following, we will mostly use curves that are parametrised by vertical length.}
\begin{itemize}
\item[c0)] $\gamma'\neq 0$,\label{item c0}
\item[c1)]  $\gamma$ has homotopy class $(0,1),$\label{item c1}
\item[c2)] $\gamma'(t)\in \fC_c,$ for each $t\in \bT.$\label{item c2}
\end{itemize}
We assume that for each $\gamma\in\Upsilon$ there exist $\{\gamma_j\}_{i=1}^N\subset \Upsilon$ such that $F^{-1}(\Upsilon)=\{\gamma_j\}_{i=1}^N$.
\vspace{5pt}

\textbf{(A4)} Let
\begin{equation}\label{def of zetar}
\zeta_r:=6(r+1)!\,.
\end{equation}
We assume that $F$ satisfies the pinching condition
\begin{equation}\label{pinching condition}
(\lambda_c^+)^{\zeta_r}<\lambda_-.
\end{equation}
\vspace{2pt}

\textbf{(A5)} With the notation \eqref{the map F} we require, for each $p\in\bT^2,$
\[
\partial_xf(p)>\max \left\{ 2(1+\|\partial_x\omega\|_\infty),|\partial_\theta f(p)|\right\}.\vspace{8pt}
\]

From now on we set 
\begin{equation}\label{eq:mumax}
\lambda_c :=\max \{ \lambda_c^+,1/\lambda^-_c\}\ge  1.\vspace{10pt}
\end{equation}

\begin{ex}
Notable examples of SVPH are partially expanding systems with neutral or close to neutral central direction which includes:
\begin{itemize}[leftmargin=*]
\item Skew torus extensions over smooth expanding map (and small perturbations of it): $F(x,\theta)=(f(x), \theta+\omega(x))$ on $\bT^2$ where $f$ is expanding. 
\item Fast-slow partially expanding systems: $F_\ve(x,\theta)=(f(x,\theta), \theta+ \ve \omega(x,\theta))$ on $\bT^2$ where $\sup_{(x,\theta)\in \bT^2}\partial_x f(x,\theta)>1$ and $\ve>0$. 
\end{itemize}
{It is easy to see that the first example satisfies the assumptions in \cite[Lemma 2.2]{CaLi} and hence, an SVPH, provided the expansion is large enough.\footnote{If not, one can always consider a finite iterate of the map to ensure of having ``enough expansion".} For a detailed discussion of the second, see \cite[Section 10]{CaLi}.}
\end{ex}

{SVPHs have finitely many exponentially mixing acips provided that \textit{most} preimages are transversal.\footnote{A property that is related to \textit{accessibility} in the hyperbolic or partially hyperbolic settings (see \cite[Remark 4.3]{CaLi}).} To state this more precisely, we recall some definitions.}
\begin{defn}[Transversality of preimages]
\label{transversality}
Given $n\in \mathbb{N}$, $y\in \mathbb{T}^2$ and $z_1,z_2 \in F^{-n}(y)$, we say that  $z_1$ is  \textit{transversal} to $z_2$ (at time $n$) if $D_{z_1} F^n \mathbf{C}_u \cap D_{z_2} F^n \mathbf{C}_u=\lbrace 0 \rbrace$, and we write $z_1 \pitchfork_n z_2$ where $\fC_u$ is the {unstable cone} for $F$. 
\end{defn}
The following function measures the growth of non-transversal preimages: 
\begin{equation*}\label{NF1}
\mathcal{N}_{F}(n)=\sup_{y\in \mathbb{T}^2} \sup_{z_1\in F^{-n}(y)}\mathcal{N}_F( n, y, z_1).
\end{equation*}
where
\begin{equation*}\label{NF}
\mathcal{N}_{F}( n, y, z_1):=\sum_{\substack{z_2 \not{\pitchfork}_n z_1 \\ z_2\in F^{-n}(y)}}  |\det D_{z_2}F^{n}|^{-1}
\end{equation*}
for each $y\in \mathbb{T}^2$ and $z_1\in F^{-n}(y)$.

A related quantity (see \cite[Lemma 7.4]{CaLi}) which has the advantage to be submultiplicative in $n$ and upper semicontinuous in $F$, is the following: 
$$\widetilde{\mathcal{N}}_{F}(n)=\sup_{y\in \mathbb{T}^2} \sup_{L}\, \widetilde{\mathcal{N}}_F(n, y, L)$$
where 
\[
 \widetilde{\mathcal{N}}_{F}(n, y, L):=\sum_{\substack{z\in F^{-n}(y) \\ DF^{n}(z)\fC_{u}\supset L }}  |\det DF^{n}(z)|^{-1}.
\]
for each $y\in \mathbb{T}^2$ and line $L$ in $\mathbb{R}^2$ passing through the origin.  
\begin{oss}
For instance, suppose we have a skew-product over a linear expanding map $F(x,y)=(\ell x, y+ \omega(x))$ on $\bT^2$, with $\ell>1$. Then the above quantities are bounded by 1, since
\[
\sum_{\substack{z_2 \not{\pitchfork}_n z_1 \\ z_2\in F^{-n}(y)}}  |\det D_{z_2}F^{n}|^{-1}\le \sum_{z_2\in F^{-n}(y)} \frac{1}{\ell^n}=\ell^{-n}\sharp\{z_2\in F^{-n}(y), y \in \bT^2\} =1.
\]
\end{oss}
\vspace{5pt}

Our last assumption on the system is the following.\footnote{As $\widetilde \cN_F(n)$ is submultiplicative, it is enough to check the condition for some integer $n_0 \in \bN$ to have the bound for the limit.} 
\begin{equation}\label{main-hp}\tag{\textbf{A6}}
\lim_{n\to \infty}\widetilde{\mathcal{N}}_{F}(n)^{\frac 1n}<1.\vspace{10pt}
\end{equation}
\begin{oss}
Under this assumption, the growth of the number of non-transversal preimages is under sufficient control. In \cite{Tsu} it is proven that, for 2D partially hyperbolic systems, condition \eqref{main-hp} is generic in the $\cC^r$ topology.
\end{oss}

Under the above assumptions we have the following results on the existence of invariant measures for the systems. The proof, in much more detail, is contained in Proposition \ref{prop:spectral}.
\begin{lem}\label{lem:invariants}
Let $F$ be a SVPH such that \eqref{main-hp} holds. Then there exist $F-$invariant absolutely continuous (w.r.t.~Lebesgue measure $\frm$) probability measures $\{\mu_k\}_{k=1}^{\ell}$ with densities $\{\rho_k\}_{k=1}^\ell$ such that
\begin{enumerate}[leftmargin=*]
\item[(i)] each $\mu_k$ is an exponentially mixing physical measure,\footnote{That is, \[\frm\left(x\in \bT^2: \frac 1n \sum_{k=0}^{n-1}\delta_{F^k(x)}\to \mu_i\right)>0.\]}
\item[(ii)] each $\rho_k$ is a non-negative function in the Hilbert space $H^s(\bT^2)$, for some integer $s>1$, 
\item[(iii)] each couple $(\mu_k,\rho_m)$ satisfies $\mu_k(\rho_m) = \delta_{k,m}$ and $\frm\left(\bT^2 \setminus \bigcup_{j=1}^\ell \supp \rho_j\right) =0\,.$
\end{enumerate}
\end{lem}

\subsection{Results}

 Let $F:\bT^2\to \bT^2$ be a map, $\tau: \bT^2 \to \bR$ be a centred observable\footnote{$\mu(\tau) = 0$ for each ergodic acip $\mu$.} and consider the ergodic sums 
 \begin{equation} \label{def:tau}\tau_n := \sum_{k=0}^{n-1} \tau \circ F^k.
 \end{equation}
 We will use the symbol $m$ to denote any initial probability measures absolutely continuous with respect to $\frm$, $f_m$ to be the associated density, i.e., $\dd m = f_m\, \dd \frm$ and by $\bP_m$ the associated probability distribution function.
 
 We say that $(F,\tau)$ satisfies a non-trivial CLT if there exist $\sigma>0$ and $A \in \bR$ such that
$$\lim_{n \to \infty }\bP_{m}\left(\frac{\tau_n}{\sqrt{n}}\leq z\right) = \int_{-\infty}^z \fn\left(\frac{y}{\sigma}\right)\, \dd y\,,\,\,\, z \in \bR$$
 and $\fn(y)=\frac{1}{\sqrt{2\pi}}e^{-y^2/2}$ is the standard normal density. 

More general CLT can be established when the initial measure is not an acip and $F$ is not ergodic but it has finitely many acips (see \cite[p.~119]{HeHe}). We recall this general form below.

Suppose $F$ has $\ell$ acips and then the general CLT, if it holds, takes the form 
\begin{equation}\label{CLT}
    \lim_{n \to \infty }\bP_{m}\left(\frac{\tau_n}{\sqrt{n}}\leq z\right) = \sum_{k=1}^\ell c_k \int_{-\infty}^z \fn\left(\frac{y}{\sigma_k}\right)\, \dd y\,,\quad \text{as}\\,,\,\,\, z \in \bR
\end{equation}
where $c_k$ and $\sigma_k$, $k=1,\dots, d$ are constants\footnote{We will find an explicit expression for the coefficients $c_k$ and $\sigma_k$ (see Theorem \ref{thmB}).} with $\sum c_k =1$ and $\sigma_k>0$. 

\begin{oss}
Intuitively, if the system is already at the equilibrium, the initial mass can be decomposed to individual ergodic basins and each component can be evolved independently. If each component evolves into a Gaussian, then the overall asymtptotic should be a convex combination of those Gaussians, weighted by how the system was initialised. For example, if the initial measure is supported in a single ergodic basin, then all $c_k$s except one will be $0$. 
\end{oss}

The Berry-Esseen estimate given below provides an elementary quantitative rate of convergence in the CLT. 
\begin{equation}\label{BEst}
    \sup_{z \in \mathbb R}\left|\bP_{m}\left(\frac{\tau_n}{\sqrt{n}}\leq z\right) - \sum_{k=1}^dc_k \int_{-\infty}^z \fn\left(\frac{y}{\sigma_k}\right)\, dy\,\right|  \leq \frac{C(\||m|\|+1)}{\sqrt{n}}
\end{equation}
where the norm $\||\cdot|\|$ is understood to be the norm in a suitable (dual) Banach space. In order to understand local statistics of $\tau_n$, we require LLTs. The following is such an LLT:
\begin{equation}\label{fLLT}
    \lim_{n\to\infty} \sup_{z\in \bR}\left|\sqrt{n} \bE_m\left(g(\tau_n - z)\right) - \frm(g)\sum_{k=1}^{\ell} c_k\fn\left(\frac{z}{\sigma_k\sqrt{n}}\right) \right| = 0
\end{equation}
where $g \in C^0(\bR)$ be bounded and $g(z) = o(z^{-2})$ as $z\to\infty$.\footnote{In particular, one can consider $g \in C^0_c(\bR)$\,.}

For these limit theorems to hold, our random variables must be non-degenerate. So, we make the following non-degeneracy assumptions on $\tau \in \cC^r(\bT^2,\bC)$.
\vspace{5pt}

\textbf{(B1)} There is no ergodic basin, $D$, of $F$ and that there are no $\Psi\in C^0(D, \bR)$, $c\in \mathbb R$ such that $\Psi\in L^1(\bT^2, \bR)$, $c\in \bR$ such that 
\[
\tau|_D=\Psi\circ F-\Psi+c.
\]

(This is a non-coboundary condition.)\vspace{5pt}

\textbf{(B2)} There is no ergodic basin, $D$, of $F$ and that there are no $\Psi\in C^0(D, \bR)$, $a>0, b\in \bR$ such that
\[
\tau|_D+\Psi-\Psi\circ F
\]

is $a\mathbb{Z}+b$ valued.

(This is a non-lattice condition.)

\begin{oss}
Note that when $F$ is ergodic these two conditions reduces to the usual non-coboundary and non-lattice conditions assumed in the literature. 
\end{oss}

We are ready to state our results, which proves are presented in Section \ref{sec:proof-thmA}. The first one is the CLT.
\begin{bigthm}\label{thmA}
Let $F$ be a SVPH satisfying \eqref{main-hp} and $\tau \in \cC^r(\bT^2, \bC)$ be such that \emph{\textbf{(B1)}} is true. Then the non-trivial CLT \eqref{CLT} and the Berry-Esseen estimate \eqref{BEst} hold.  
\end{bigthm}

\begin{oss}
From the proof it follows that, while the constants $c_k$s depend on the initial measure $m$, the constants $\sigma_k$ (aka diffusion coefficients) do not, i.e., they depend only on $F$ and $\tau$ and are independent of how the system is initialised.
\end{oss}

\begin{oss}
If $m$ is an invariant acip (not necessarily ergodic), then one can establish more general limit theorems where both $\tau_n$ and $\tau \circ F^n$ are tracked, e.g., the limiting behaviour of $\bE_m(f(\tau \circ F^n)g(\tau_n/\sqrt{n}))$ and $\bE_m(f(\tau \circ F^n)g(\tau_n-z))$ where $f$ and $g$ are appropriately chosen. We refer the reader to \cite{HeHe} for the exact statements. 
\end{oss}

To prove Theorem \ref{thmA} we will use the functional analytic framework presented in Appendix \ref{app:AbstractCLT} applied to the transfer operator $\cL_F$ associated to $F$ and the \textit{twisted transfer operators} $\cL_\nu,\,\,\, \nu \in \bR$. This will be the context of section \ref{sec:TraOp}. In section \ref{sec:twisted}, we do a further analysis of the spectral properties of $\cL_\nu$ to obtain the following LLTs, which proof can be found at the end of section \ref{sec:proof-thmA}. 

\begin{bigthm}\label{thmB}
Let $F$ be a SVPH satisfying \eqref{main-hp} and $\tau \in \cC^r(\bT^2, \bC)$ be such that \emph{\textbf{(B2)}} holds. Then the LLT \eqref{fLLT} holds for any initial measures $\dd m=f_m \dd\frm$. Also, there exists $C>0$ such that for all $a, b \in \bR$ with $0<b-a<1$ and for each $\delta>2$,
\begin{equation*}\label{LLT}
    \left|\bP_{m}\left(\frac{\tau_n}{\sqrt{n}}\in [a,b]\right) - \sum_{k=1}^dc_k \int_{a}^b \fn\left(\frac{y}{\sigma_k}\right)\, dy\,\right| \leq C\left(\frac{\max_k  \fn\big(\frac{a}{\sigma_k}\big)}{n^{\frac 12-\frac 1\delta}}+\frac{b-a}{\sqrt{n}}\right)\,,
\end{equation*}
where, recalling $\mu_k$ and $\rho_k$ from Lemma \ref{lem:invariants}, $c_k = \frm(\rho_k)\mu_k(f_m)$ and
$$\sigma^2_k = \mu_k(\tau_k^2)+2\sum_{j=1}^\infty \mu_k(\tau_k\cdot\tau_k\circ F^j)\,.$$
\end{bigthm}

\begin{oss}
To prove the second part of the above theorem, we follow the arguments given in \cite[Section 1.5]{DeKiLi} for expanding maps adapted to our non-ergodic partially hyperbolic setting. It is easy to note that the usefulness of the result depends on the size of the interval: if $b-a=O(n^{-\gamma})$ there is an interplay between $\delta$ and $\gamma$ in order to get a useful estimate, while if $a$ or $b$ are $O(\sqrt{\log n})$ we gain no information from the previous estimate. To get results more close to optimal, we would need an Edgeworth expansion (following the ideas in \cite{FeLi, FePe}) which will be established in a subsequent project.
\end{oss}


\section{Unperturbed transfer operator}\label{sec:TraOp}
\subsection{Quasi-compactness}
In this section we explore the spectral properties of the (SRB)-transfer operator associated to $F$ and defined via
\begin{equation}\label{L_F}
    \cL_F u(x)=\sum_{y\in F^{-1}(x)} \frac{u(y)}{|\det DF(y)|}, \quad u\in L^1(\bT^2),\; x,y \in \bT^2,
\end{equation}
which is the main tool to study the statistical properties of $F$.

The main technical result of \cite{CaLi} is a Doeblin-Fortet-Lasota-Yorke (DFLY) inequality in suitable functional spaces, which implies quasi-compactness for the operator $\cL_F$, and it will be crucial for our arguments. To state this result, we recall the definition of the Banach spaces used. 

For a positive integer $s>1$ to be fixed later, $H^s(\bT^2, \bC)$ will denote the standard Hilbert space of $L^2(\bT^2, \bC)$ functions whose $s$-th weak derivative belongs to $L^2$, endowed with the norm\footnote{This is not the classical norm used in such a space, but it is equivalent to it and we will take this formulation for simplicity.}
\begin{equation}
\label{equivalence of norm H L}
\|u\|^2_{H^s}:= \sum_{\gamma+\beta=s} \| \partial_{x_1}^\gamma \partial_{x_2}^{\beta} u \|^2_{L^{2}}.
\end{equation}
We will denote by $\cB_s$ the completion of $\cC^{r-1}(\bT^2, \bR)$ with respect to the norm
\[
\|u\|_s=\|u\|_{H^s}+\|u\|^*_{s+2}
\]
where
\begin{equation}
\label{geom norm}
\| u \|^*_{\rho}:=\max_{|\alpha|\le \rho}\sup_{\gamma \in \Gamma} \sup_{\substack{\phi \in \mathcal{C}^{|\alpha|}(\mathbb{T}, \bR)\\\|\phi\|_{\mathcal{C}^{|\alpha|}}=1}} \int_{\mathbb{T}} \phi(t)(\partial ^{\alpha} u)(\gamma(t))dt, \quad \rho\in \bN
\end{equation}
where $\Gamma$ is a set of admissible curves in the central cone (recall assumption {\textbf{(A3)}). Note that the completion of $\cC^r$ with respect to the norm $\|\cdot\|^*_\rho$ defines a decreasing sequence of Banach spaces continuously embedded in $L^1$, namely
\begin{equation}
\label{relation norm rho rho' and L1}
\|u\|_{L^1}\le {C}\|u\|^*_{\rho_1}\le {C}\|u\|^*_{\rho_2}, \qquad \mbox{for every} \quad 0\le \rho_1\le \rho_2\le r-1.
\end{equation}
$(\cB_{s}, \|\cdot \|_s)$ is, in fact, a Banach algebra satisifying the following.
\begin{itemize}[leftmargin=20pt]
\item $\cC^{r-1}(\bT^2)\subset \cB_{s}\subset H^{s}(\bT^{2})\subset \cC^0(\bT^2)$ where the first inclusion is dense and the second follows because $s>1$.
\item $\cB_s$ is compactly embedded in $L^1(\bT^2)$. 
\end{itemize}
For further details, we refer the reader to \cite{CaLi}.

We also recall the following DFLY inequality.\footnote{The statement of the result  has been formulated here in a simpler way to accommodate what is essentially needed in the present work.}
\begin{thm}[Quasi-compactness of \texorpdfstring{$\cL_F$}{Lg} \cite{CaLi}]\label{thm:LY-F} Let $r \ge 4$ and $F\in \cC^r(\bT^2,\bT^2)$ be an SVPH satisfying \eqref{main-hp}. Then there exist $1< s\le r-3$, $\sigma_* \in (0,\lambda^{-1}\lambda_c^{\xi_s})$, $C_\sigma$ and $D_\sigma>0$ such that the (SRB)-transfer operator $\cL_F:\cB_{s}\to\cB_{s}$ satisfies\footnote{We will use $L^1$ as a short for $L^1(\bT^2)$.}
\begin{equation}\label{LY-F}
\begin{split}
&\|\cL_F^n u\|_{L^1}\le \|u\|_{L^1}\\
&\|\cL_F^n u\|_{s}\le C_{\sigma_*} \sigma_*^n \|u\|_s+D_\sigma \|u\|_{L^1}
\end{split}
\end{equation}
for each $n \in \bN$ and $u\in \cC^{r-1}(\bT^2,\bC)$.
\end{thm}

From now on we fix $s>1$ as in Theorem \ref{thm:LY-F}.
The main consequences of the above theorem are summarized in the following proposition.

\begin{prop}\label{prop:spectral}
The following is true.
    \begin{enumerate}[leftmargin=*]
    \item[(i)]  $\cL_F$ is quasi-compact  when acting in $\cB_s$, with spectral radius $r(\cL_F)=1$ and essential spectral radius smaller than $1$,\footnote{That is $\cB_s = \cF \oplus \cH$ with $\cF,\cH$ subspaces $\cL_F$-invariant, $r(\cL_{F}|_{\cH}) < 1$, $\dim \cF < \infty$, and $r(\cL_{F}|_{\cH})< |\zeta| \le 1$.} and $\sup_n \|\cL_F^n\|_{L^1} < \infty\,.$
    \item[(ii)] There exist $F-$invariant probability measures $\mu_1, \dots,\mu_\ell$ in $\cB_s^\prime$ such that $\mu_k(\rho_m) = \delta_{k,m}$, each one being an exponentially mixing physical measure, 
    where $\rho_1,\dots,\rho_\ell$ are non-negative eigen-functions in $\cB_s$  that forms a basis for Ker$(\cL_F - \Id) \cap \cB_s\,$ and $\frm\left(\bT^2 \setminus \bigcup_{j=1}^\ell \supp \rho_j\right) =0\,.$
\end{enumerate}
\end{prop}
\begin{proof} The proof follows \cite[Appendix B]{BaGo} 
where piece-wise smooth hyperbolic maps are considered. We sketch the adapted proof here for the convenience of the reader.

The quasi-compactness of the operator $\cL_F$  and the existence of finitely many invariant probabilities follows from Theorem \eqref{thm:LY-F} and H\'ennion theorem \cite{Hen}. 
From \cite{CaLi} we also know that the eigenvalues of modulous 1 form a finite group, and their geometric and algebraic multiplicities are both finite. Let $h$ be such that $\cL_F h=h $ and $\Pi$ be the projector on the eigenspace $\cE_1$ corresponding to eigenvalue 1. 
So, $\Pi \varphi=\lim_{n} \frac 1n \sum_{k=0}^{n-1}\cL_F^k \varphi$ and the convergence holds in $\cB_s$. Moreover, letting $\dd \frm$ being the Lebesgue measure on $\bT^2$, for each $\varphi \in  \cB_s$ and $g\in L^1$, we have 
\[
\bigg|\int g\cdot\Pi \varphi \, \dd \frm\bigg| \le \|\varphi\|_{\infty} \bigg| \int g\cdot\Pi 1  \, \dd \frm \bigg|. 
\]
In particular, $h \, \dd \frm$ is absolutely continuous with respect to $$ \Pi 1 \, \dd \frm=: \mu_*\,.$$  Let us consider the space 
\[
\bV=\{\varphi \in L^\infty(\bT^2) : \varphi \ge 0, \;\varphi\circ F=\varphi\}
\]
and the map $\Phi: \cE_1\to \bV$ defined by $\Phi (h)=h_*$ where $h_*$ is the Radon-Nikodym derivative of the measure $h \, \dd \frm$ with respect to $\mu_*$. 
Since $\cC^r$ is dense in $\cB_s$, this map is an isomorphism (see \cite[Appendix B]{BaGo} for the details), and in particular, given any $h\in \cE_1$ with $h_* \in \bV$, the function $h_* \Pi 1 \in \cB_s\subset H^s$. 

Next, let 
\[
\bV_1=\{\varphi \in \bV : \mu_*(\varphi)=1\} \subset \bV.
\]

Then $\bV_1$ is a (non empty) convex cone in $\bV$ whose extremal points are of the form $\mathds{1}_E$ for some minimal invariant set $E$. Such extremal points are automatically linearly independent. Since $\operatorname{dim} \bV_1 < \infty$, there is only a finite number of them, say $\mathds{1}_{E_1},..., \mathds{1}_{E_\ell}$, and a function $\varphi$ belongs to $\bV_1$ if and only if it can be written as
$\varphi=\sum_{j}\alpha_j \mathds{1}_{E_j}$ for some scalars $\alpha_1,..,\alpha_\ell$. The decomposition of the function $1\in \bV_1$
is given by $1 = \sum_j\mathds{1}_{E_j}$, hence the sets $E_j$ cover the whole space up to a set of zero $\mu_*-$measure.
It follows that $\dfrac{\mathds{1}_{E_j}}{\mu_*(E_j)}\in \bV_1$ and  the measures  $$\mu_j:= \dfrac{\mathds{1}_{E_j}\cdot \mu_*}{\mu_*(E_j)}$$ are $F$-invariant ergodic probability measures.

Moreover, from the previous result, the functions \begin{equation}\label{eq:rhoi}\rho_j:= \mathds{1}_{E_j} \Pi 1 \in H^s
\end{equation}
and any element of $\bV_1$ is a linear combination of the $\rho_j$. Consequently, for each $\phi \in \cC^1$
\[
\begin{split}
\int \Pi(\phi  \rho_j) \mathds{1}_{E_m} \, \dd \frm& =\lim _{n \rightarrow+\infty} \frac{1}{n} \sum_{k=0}^{n-1} \int \mathcal{L}_F^k\left(\phi\rho_j\right) \cdot \mathds{1}_{E_m} \mathrm{\dd \frm} \\
& =\lim _{n \rightarrow+\infty} \frac{1}{n} \sum_{k=0}^{n-1} \int \phi \cdot \mathds{1}_{E_j} \cdot \mathds{1}_{E_m} \circ F^k \mathrm{~d} \mu_* \\
&=\int_{E_j \cap E_m} \phi \mathrm{~d} \mu_*=\mu_*\left(E_j\right) \delta_{j, m} \int \phi \mathrm{~d} \mu_m
\end{split}
\]
which implies that
\[
\Pi\left(\phi \rho_j\right)=\left(\int \phi \mathrm{~d} \mu_j\right) \rho_j.
\]
We can now conclude the exponential mixing statement: let $\varphi,\phi \in \cC^1(\bT^2)$ (for H\"older observables we can use a standard density argument), then for some $\ve \in (0,1)$,

\[
\begin{split}
\int \phi \cdot \varphi \circ F^n \, \dd \mu_j&=\frac{1}{\mu_*(E_j)} \int \cL^n(\phi \rho_j) \varphi \, \dd \frm\\
&=\frac{1}{\mu_*(E_j)} \int (\cL^n-\Pi)(\phi \rho_j) \varphi \, \dd \frm+\frac{1}{\mu_*(E_j)} \int \Pi(\phi \rho_j) \varphi \, \dd \frm\\
&= O(\varepsilon^n)+ \frac{1}{\mu_*(E_j)} \int \left(\int \phi \mathrm{~d} \mu_j\right) \rho_j \varphi \, \dd \frm\\
&=O(\varepsilon^n)+ \left(\int \phi \mathrm{~d} \mu_j\right) \int  \varphi \cdot\frac{1}{\mu_*(E_j)} \mathds{1}_{E_j} \Pi 1  \, \dd \frm\\
&= O(\varepsilon^n)+ \left(\int \phi \mathrm{~d} \mu_j\right) \int  \varphi \cdot\frac{1}{\mu_*(E_j)} \mathds{1}_{E_j} \, \dd \mu_*\\
&=  O(\varepsilon^n)+ \left(\int \phi \mathrm{~d} \mu_j\right) \int  \varphi \, \dd \mu_j,
\end{split}
\]
which proves that each $\mu_j$ is exponentially mixing.

Finally, we claim that there is a partition 
of all $\bT^2$ (up to a $\frm-$null set) into ergodic basins of the physical measures of $\mu_j$. In particular, there are $D_j \subset \bT^2$ such that $\frm(D_j\cap D_j)=0$ for $i \neq j$, $\frm(\sum_j \mathds{1}_{D_j})=1$ and, for each $f\in \cB_s$,
\begin{equation}\label{eq:projection}
\Pi(f)=\sum_{j} \frac{1}{\mu_*(E_j)}\frm(f \mathds{1}_{D_j})\rho_j,
\end{equation}
where the $\rho_j$ are defined in \eqref{eq:rhoi}. This prove the last statement of the proposition.

Let us suppose that $\Pi(f)=\sum_{j} \xi(f)\rho_j$ and let us find out who are the coefficients $\xi(f)$. Let us fix some $j$ and let us set
\[
h_{j,n}:=\lim_n \frac 1n \sum_{j=0}^{n-1}\mathds{1}_{E_j}\circ F^k \subset L^2(\bT^2).
\]
By the invariance of $E_j$,
\[
\lim_n \frm(h_{j,n}\cdot  f)=\langle \mathds{1}_{E_j}, \Pi(f) \rangle_{L^2}=\xi_j(f)\mu_*(E_j).
\]
Let $h_j^*$ be a weak limit of $h_{j,n}$. From the previous computation, we see that $h_j^*$ is $F$-invariant and such that
\begin{itemize}
    \item[(a)] $\xi_j(f)=\dfrac{\frm(h_j^* \cdot f)}{\mu_*(E_j)}$ 
    \item[(b)] $\xi_j(1)=1 \implies \frm(h^*_j)=\mu_*(E_j)$.
\end{itemize}
Moreover, again from invariance, for each $k$
\[
\begin{split}
\frm(h^*_k\cdot  h^*_j)&=\mu_*(E_j)\lim_n \xi_j\left(\mathds{1}_{E_k} \frac{1}{n}\sum_{r=0}^{n-1}\cL_F^r1\right)\\
&=\mu_*(E_j)\xi_j(\rho_k)=
\begin{cases} \mu_*(E_j) \qquad &\text{if} \quad j=k \\ 0 \qquad &\text{if} \quad j\neq k \end{cases}.
\end{split}
\]
It follows that $\frm((h^*_j)^2)=\mu_*(E_j)=\frm(h^*_j)$ which implies the existence of $D_j\in \bT^2$ such that
\begin{itemize}
    \item[(c)] $h_j^*=\mathds{1}_{D_j}$
    \item[(d)] $\frm(\mathds{1}_{D_j})=\mu_*(E_j)$
    \item[(e)] $\frm(D_j \cap D_k)=0$ for $j\neq k$.
\end{itemize}
The claim redly follows from (a)--(e).
\end{proof}

\section{Perturbed transfer operators}\label{sec:twisted}
In order to deal with complex transfer operators, for any real Banach space $(\cB, \|\cdot \|)$ defined in this paper we use the auxiliary norm $\|\cdot\|_{\cB,\bR}$ and the norm $\|\cdot\|_{\cB,\bC}$ (see Appendix \ref{app:complex}) to define the complexification $(\cB, \|\cdot\|_{\cB,\bC})$ of $(\cB, \|\cdot \|)$. 

If $u \in \bC$ with $u=\re (u) + i \im  (u)$, by equations \eqref{eq: complex} the DFLY inequalities hold in the space $(\cB_s, \|\cdot\|_{s,\bC})$, where $\|u\|_{s,\bC}=\|u\|_{H^s}+\|u\|^*_{s+2, \bC}$ and
\begin{equation}\label{eq:def-complex-norm}
    \|u\|^*_{\rho,\bC}:=\sup_{\theta \in [0,2\pi]} \left|e^{i\theta}\star \big(\re (u) , \im  (u)\big)\right|
\end{equation}
where the product $\star $ is defined in \Cref{app:complex}.
\begin{oss}\label{rmk:confusion}
Henceforth, when the function $u$ is clear from the context, we will not specify if we are using the complex or real norms, and the same symbol will be used for both, depending whether $u\in \cC^r(\bT^2,\bR)$ or $u\in \cC^r(\bT^2,\bC)$.
\end{oss}


Next, we define twisted (or perturbed) transfer operators. 

\begin{defn}
 Given a local diffeomorphism $ F\in \cC^r(\bT^2,\bT^2)$ and $\tau \in \cC^r(\bT^2, \bR)\,,$ for any $\nu \in \bR$ and $u\in L^1(\bT^2, \bC)$ we define the twisted transfer operator $\cL_\nu$ associated to $F$ by
\[
\cL_\nu u (x)=	\sum_{y \in F^{-1}(x)} \dfrac{u(y)}{\det DF(y)} e^{i\nu \tau(y)}.
\]
\end{defn}Iterating the above formula we see that for each $n\in \bN$
\begin{equation}\label{L^n}
\cL^n_\nu u (y)=\sum_{y\in F^{-n}(x)} \dfrac{u(y)}{\det DF^n(y)} e^{i\nu \tau_n(y)}.
\end{equation}
where is defined in \eqref{def:tau}.

Our results rely on the following two Lemmas on the quasi-compactness of the twisted transfer operator.

\begin{lem}[Quasi-compactness of \texorpdfstring{$\cL_\nu$}{Lg}]\label{lem:LYnu}
There exists $\sigma\in (0,1)$ such that, for each $\nu \neq 0$, the spectrum of $\cL_\nu$ acting on $\cB_s$ is contained in the disk $\{z\in \bC: |z|\le 1\}$ and the essential spectrum is contained in the disk $\{z\in \bC: |z|\le \sigma\}$. 
\end{lem}

\begin{proof}
First of all we see that $\|\cL_\nu\|_{L^1}\le \|\cL_F\|_{L^1}\le 1\,.$
Therefore, the spectrum of $\cL_\nu$ is contained in the disk $\{z\in \bC: |z|\le 1|\}$. 

We claim that there exists $q_* \in \bN$ and $\sigma \in (0,1)$ there is $C>0$ such that, for each $\nu \neq 0$, for all $n\in\bN$ and $ u\in \cB_s$,
\begin{equation}\label{LYnu}
\|\cL_\nu^n u\|_{s}\le C(1+|\nu|)^{q_*}\left(\sigma^n\|u\|_{s}+\| u \|_{L^{1}}\right).
\end{equation}
Thanks to the H\'ennion theorem (see \cite{Hen}), the above DFLY inequality proves that the essential spectral radius of $\cL_\nu$ in $\cB_s$ is strictly smaller than $1$. 

Let us prove the claim. By the definition od the norm $\|\cdot\|_s$ we will estimate the norms $\|\cL_\nu u\|^*_{s+2}$ and $\|\cL_\nu u\|_{H^s}$ separately in two steps and then we merge the estimates in the final step.\\

\noindent\textit{Step 1} (estimate of $\|\cL_\nu u\|^*_{s+2}$). Recall that the complex norm $\|\cdot\|^*_{s+2}$ is defined via \eqref{eq:def-complex-norm}.
For any $q\in \bN$, $\alpha \le q$, $\phi \in \cC^{|\alpha|}(\bT)$ and  $\gamma \in \Gamma$, due to \eqref{commute}, we have
\begin{equation}\label{eq:J1+J2}
\begin{split}
&\int \phi\, (\partial^\alpha \cL^n_\nu u)\circ\gamma \\ &=\int \phi\,  [\cL_F^n(e^{i\nu \tau_n}P^{\alpha}_{n,q}u)]\circ\gamma \,\\&\qquad\qquad+\int \phi \, \Big[\cL^n_F\Big( e^{i\nu\tau_n}\sum_{k=1}^s (i\nu)^k\sum_{j=1}^{m_k}  a_{j, k}(\cD^u_{j,k}u)^{\beta_{j,k}}(\cD^\tau_{j,k}\tau_n)^{\gamma_{j,k}}\Big)\Big]\circ\gamma\\
&= \cJ_1+\cJ_2
\end{split}
\end{equation}
where 
$$\cJ_1:=\int \phi\,  [\cL_F^n(e^{i\nu \tau_n}P^{\alpha}_{n,q}u)]\circ\gamma$$
and 
$$\cJ_2:=\int \phi \, \Big[\cL^n_F\Big( e^{i\nu\tau_n}\sum_{k=1}^s (i\nu)^k\sum_{j=1}^{m_k}  a_{j, k}(\cD^u_{j,k}u)^{\beta_{j,k}}(\cD^\tau_{j,k}\tau_n)^{\gamma_{j,k}}\Big)\Big]\circ\gamma.$$
By Lemma \eqref{lem:commute} the term $\cJ_2$ involved only derivatives of $u$ up to order $\alpha-1\leq q-1$ so that, using the last statement of Lemma \ref{lem:commute}, equations \eqref{eq: complex} and \eqref{L-Y B2}, we can bound real part and imaginary part of $\cJ_2$ with  
$
C_{s,\tau}\lambda_c^n(1+|\nu|)^q \| u \|^*_{q-1}
$
for some constant $C_{s,\tau}>0$. 

Let us estimate $\cJ_1$. Setting
\[
x=\operatorname{Re}(P^\alpha_{n,q}u)  \qquad \mathrm{and} \qquad y=\operatorname{Im}(P^\alpha_{n,q}u),
\]
to be respectively the real part and the imaginary part of $P^\alpha_{n,q}u$, by \eqref{eq: complex}\footnote{Recall Remark \ref{rmk:confusion} for the following abuse of notation: for $f$ real, $\|f\|^*_0$ is the norm defined in \eqref{geom norm}, while if $f$ is complex valued, then $\|f\|^*_0$ is defined analogously to \eqref{eq:def-complex-norm}.}
\begin{equation}\label{eq:sqrt}
\begin{split}
|\re \cJ_1|&\le \|\cL^n_F(x e^{i\nu \tau_n})\|^*_0+\|\cL^n_F(y e^{i\nu \tau_n})\|^*_0\\
&\le \left[(\|\re(\cL^n_F(x e^{i\nu \tau_n})\|^*_0)^2+(\|\im(\cL^n_F(x e^{i\nu \tau_n})\|^*_0)^2\right]^{\frac 12}\\
&\qquad\qquad+\left[\|\re(\cL^n_F(y e^{i\nu \tau_n})\|^*_0)^2+(\|\im(\cL^n_F(y e^{i\nu \tau_n})\|^*_0)^2\right]^{\frac 12}.
\end{split}
\end{equation}
We thus want to estimate each of the terms in the square roots. Since for any $\phi \in \cC^q$ with $\|\phi\|_{\cC^q}\le 1$, any  $u\in \cC^r(\bT^2,\bR)$ and any $f \in \cC^0(\bT^2,\bR)$, we have
\[
\begin{split}
\int \phi \, \cL_F^n(f P^{\alpha}_{n,q}u)\circ\gamma%
&=\sum_{\tilde \gamma \in F^{-n}\gamma}\int \phi \, \frac{(f P^{\alpha}_{n,q}u)\circ\tilde\gamma }{J_n \circ \tilde \gamma}\\%
&\le \sup_{\psi \in \cC^0: \|\psi\|_{\cC^0}\le 1} \int \psi \sum_{\tilde \gamma \in F^{-n}\gamma}\frac{(fP^{\alpha}_{n,q}u)\circ\tilde\gamma }{J_n \circ \tilde \gamma}\le \|f\|_{\infty} \|\cL_F^n(P^\alpha_{n,q}u)\|_0^*\,,
\end{split}
\]
and
\[
\|\re(\cL^n_F(x e^{i\nu \tau_n})\|^*_0 =\|\cL_F^n(x\cos(\nu \tau_n))\|^*_0\le \|\cL_F^n(x)\|^*_0.
\]
Obviously, an analogous estimate is true for all the other terms in the square roots in \eqref{eq:sqrt}, from which we deduce, using \eqref{eq: complex} again, that
\[
\begin{split}
|\re \cJ_1 | &\le [(\|\cL^n_F(x)\|^*_0)^2+(\|\cL^n_F(x)\|
^*_0)^2]^{1/2}+[(\|\cL^n_F(y)\|^*_0)^2+(\|\cL^n_F(y)\|
^*_0)^2]^{1/2}\\
&=\sqrt 2 (\|\cL^n_F(x)\|^*_0+\|\cL^n_F(y)\|^*_0)\\
&\le 2 \|\cL^n_F(x+iy)\|^*_0\\
&=2 \|\cL_F^n(P^\alpha_{n,q}u)\|^*_0.
\end{split}
\]

Clearly, the same estimate holds for $\im \cJ_1$ and, recalling \eqref{eq: LnP=PLn}, we thus have
\[
\sqrt{(\re(\cJ_1)^2+(\im(\cJ_1)^2}\le  2\|\cL_F^n(P^\alpha_{n,q}u)\|_0^*=2\|\partial^\alpha\cL_F^n u\|_0^*\le C\|\cL_F^n u\|_q^*.
\]

Taking the sup over $\phi\in \cC^{|\alpha|}$, $\gamma \in \Gamma$, and $|\alpha|\le q$ in \eqref{eq:J1+J2}, by the above estimates we have\footnote{We also use the elementary inequality $\sup \sqrt{a^2+b^2}\le \sqrt{(\sup a)^2+(\sup b)^2}$.}
\begin{equation}\label{LYstar}
\|\cL_\nu^n u\|^*_{q}\le C\|\mathcal{L}_F^n u\|^*_{q}+C_{s,\tau}\lambda_c^n(1+|\nu|)^{q} \| u \|^*_{q-1}.
\end{equation}
Note that the above estimates are easier in the case $q=0$ and give, using the first of \eqref{L-Y B2}, also the weak inequality
\begin{equation}\label{eq:weak-special}
    \|\cL_\nu^n u\|^*_{0}\le C\|\cL_F^n u\|^*_0\le C\lambda_c^n \|u\|^*_0. 
\end{equation}
By the second part of \eqref{L-Y B2} (recall also that $\|\cdot\|_0 \le \|\cdot \|_{q-1}$) we have
\[
\|\cL_\nu^n u\|^*_{q}\le C\delta_*^n\| u\|^*_{q}+C_{s,\tau}\lambda_c^n(1+|\nu|)^{q} \| u \|^*_{q-1}.
\]
{Let us fix $n_0$ such that $C\delta_*^{n_0}=: \sigma_0 
<1$ and $
\sigma_0\lambda_c<1$, then the above inequality gives
\[
\|\cL_\nu^{n_0} u\|^*_{q}\le 
\sigma_0^{n_0}\| u\|^*_{q}+C_{s,\tau,n_0}(1+|\nu|)^{q} \| u \|^*_{q-1}.
\]
Therefore, writing any $n\in \bN$ as $n=kn_0+m$ with $0\le m\le n_0-1$, iterating the last inequality, using \eqref{eq:weak-special}, and slightly increasing the value of $\sigma_0$ to $\sigma_1$}, we obtain

\begin{equation}\label{eq:star1}
\begin{split}
\|\cL_\nu^{n} u\|^*_{q}&=\|\cL_\nu^{n_0}(\cL_\nu^{n-n_0}u)\|^*_{q} \le 
{\sigma_0^{n_0}}\|\cL_\nu^{n-n_0}u\| ^*_q+C(1+|\nu|)^q \|\cL_\nu^{n-n_0}u\|^*_{q-1}\\
&\le C\sigma_1^n \|u\|^*_q+C(1+|\nu|)^q \sum_{j=0}^{k-1}(\sigma_1 \lambda_c)^{jn_0}\|u\|^*_{q-1}\\
&\le C\sigma_1^n \|u\|^*_q+C(1+|\nu|)^q \lambda_c^n \|u\|^*_{q-1}.
\end{split}
\end{equation}
 Next, for reasons which will be clear soon, we want the $\|\cdot\|^*_0$- norm to be the weak norm. We can achieve that by iterating the above inequality $q$ times. We get:
 \[
\|\cL_\nu^{qn} u\|^*_{q}\le C_q \sum_{j=0}^{q-1}(\sigma_1^{q-j}\lambda_c^j)^n (1+|\nu|)^{\sum_{i=0}^{j}q-i}\|u\|^*_q+ C_q(1+|\nu|)^{\sum_{i=0}^{q-1}q-i} \lambda_c^{qn}\|u\|
^*_0. 
\]
Since $\sigma_1 \lambda^{\xi(q)}<1$ by assumption \eqref{pinching condition}, the worst case in the first term is given by $j=q-1$, so that
\[
\|\cL_\nu^{qn} u\|^*_{q}\le C_q (1+|\nu|)^{\frac{q(q+1)}{2}}\left[ (\sigma_1\lambda_c^{q-1})^n \|u\|^*_q+ \lambda_c^{qn}\|u\|
^*_0\right]. 
\]
Writing any $n\in \bN$ as $n=kq+m$ with $0\leq m\leq q-1$, and using the above inequality with $qk$ and the first part of \eqref{L-Y B2} we obtain
\begin{equation}\label{eq:LYgood}
\|\cL_\nu^{n} u\|^*_{q}\le C_q (1+|\nu|)^{\frac{q(q+1)}{2}}\left[ (\sigma_1\lambda_c^{q-1})^n \|u\|^*_q+ \lambda_c^{n}\|u\|
^*_0 \right], \qquad \forall n\in \bN. 
\end{equation}
This concludes the first step.\\

\noindent\textit{Step 2} (estimate of $\|\cL^n_\nu u\|_{H^s}$).
Let us fix $n_0\in \bN$. By inequality \eqref{LY-HB0}, 
\begin{equation}\label{eq:hatsig}
\begin{split}
\|\cL_\nu^{2n_0}u\|_{H^s}&=\|\cL_F^{n_0}(e^{i\nu\tau_{n_0} }\cL_\nu^{n_0} u)\|_{H^s}\\
&\le C \hat \sigma^{n_0} \|e^{i\nu\tau_{n_0} }\cL^{n_0}_\nu u\|_{H^s}+C\|e^{i\nu\tau_{n_0}} \cL_\nu^{n_0} u\|^*_{s+2}.
\end{split}
\end{equation}
Let us estimate the two norms in the right hand side above. By \eqref{LY Hs and Hs-1} and equation \eqref{commute} applied with\footnote{Note that, when $F$ is the identity map $\operatorname{Id}$, the operators $P_{n,s}^\alpha$ are just the derivatives.} $F=\operatorname{Id}$ we have
\[
\begin{split}
\|e^{i\nu\tau_{n_0} }\cL^{n_0}_\nu u\|_{H^s}
&\le C\lambda_c^{(2s+1){n_0}}\|e^{i\nu\tau_{n_0}}u\|_{H^s}\\
&\le C\lambda_c^{(2s+1){n_0}}\|u\|_{H^s}+ C^{n_0}(1+|\nu|)^s\|u\|_{H^{s-1}}.
\end{split}
\]
Since (\cite[Lemma E.1]{CaLi}) for each $\ve>0$ we can find $C_\ve$ such that
\[
\|u\|_{H^{s-1}}\le C \ve \|u\|_{H^s}+C_\ve \|u\|_{L^1},
\]
taking $\ve$ small enough we get
\[
\|e^{i\nu\tau_{n_0} }\cL^{n_0}_\nu u\|_{H^s}\le C\lambda_c^{(2s+1){n_0}}\|u\|_{H^s}+C_{n_0}(1+|\nu|)^s\|u\|_{L^1}.
\]

Similarly, for the second term in \eqref{eq:hatsig} we first  exploit \eqref{commute} and \eqref{LYstar} for $F=\operatorname{Id}$ and $u$ replaced by $\cL^{n_0}_\nu u$, and then we use inequality \eqref{eq:LYgood}, to get
\[
\begin{split}
\|e^{i\nu\tau_{n_0}} \cL_\nu^{n_0} u\|^*_{s+2}
&\le C (1+|\nu|)^{(s+2)(s+3)/2}[(\sigma_1\lambda_c^{s+2})^{n_0}\|u\|^*_{s+2}+C_{n_0}\|u\|^*_{0}],
\end{split}
\]
 Inserting the two above estimates into \eqref{eq:hatsig}, for each $s$ we find $\sigma_2\in (0,1)$ and $q_1>0$ such that
\[
\|\cL_\nu^{2{n_0}}u\|_{H^s}\le C(1+|\nu|)^{q_1}[\sigma_2^{n_0}(\|u\|_{H^s}+\|u\|^*_{s+2})+C_{n_0}(\|u\|^*_0+\|u\|_{L^1})].
\]
Writing any $n=kn_0+m$ and iterating the above inequality yields
\begin{equation}\label{eq:LYgood2}
\begin{split}
\|\cL_\nu^{{n}}u\|_{H^s}&\le C(1+|\nu|)^{q_2}[\sigma_2^{n}(\|u\|_{H^s}+\|u\|^*_{s+2})+C(\|u\|^*_0+\|u\|_{L^1})]\\
&= C(1+|\nu|)^{q_2}[\sigma_2^{n}\|u\|_s+C\lambda_c^n(\|u\|^*_0+\|u\|_{L^1})],
\end{split}
\end{equation}
for some $q_2>0$ depending on $s$. This concludes the second step.\\

\noindent\textit{Step 3} (Merging the estimates). By \eqref{eq:LYgood} and \eqref{eq:LYgood2} we obtained:
\[
\begin{split}
\|\cL^n_\nu u\|_s& =\|\cL^n_\nu u\|_{H^s}+\|\cL^n_\nu u\|_{s+2}^*\\
&\le   C_q (1+|\nu|)^{\frac{(s+2)(s+3)}{2}}\left[ (\sigma_1\lambda_c^{q-1})^n \|u\|^*_q+ \lambda_c^{n}\|u\|
^*_0 \right]\\
&+C(1+|\nu|)^{q_2}[\sigma_2^{n}\|u\|_s+C\lambda_c^n(\|u\|^*_0+\|u\|_{L^1})]\\
&\le C(1+|\nu|)^{q_*}[\sigma_3^{n}\|u\|_s+C\lambda_c^n(\|u\|^*_0+\|u\|_{L^1})] 
\end{split}
\]
for some $q_*>0$ and $\sigma_3 \in (0,1)$.
Finally, we want to compare the $\|\cdot\|^*_0$-norm with the $L^1$ norm. We use the following (see \cite[(9.11)]{CaLi}) which holds for each $\ell>0$:\[
\|u\|^*_0\leq \ell^{-1} \|u\|_{L^1}+ \frac {2\ell^{\frac 12}}3\|u\|_{H^1}.
\]
Choosing $\ell$ sufficiently small and possibly slightly changing $\sigma_3$ and $C$, we then obtain
\begin{equation*}\label{LYB}
\|\cL^n_\nu u\|_s \le  C(1+|\nu|)^{q_*}[\sigma^{n}\|u\|_s+C\lambda_c^n(\|u\|_{L^1})]. 
\end{equation*} 
for some $\sigma\in (0,1)$. 
\end{proof}

\begin{lem}\label{lem:NoUnitEigen}
In addition to the assumptions of \Cref{lem:LYnu}, if \emph{\textbf{(B2)}} holds, then the spectrum of $\cL_\nu$ is contained in the open disk $\{z\in \bC: |z|<1\}$.   
\end{lem}

\begin{proof}
Assume for contradiction that there exist $\theta$ and $h \in \cB_s$ (hence, continuous) such that $\cL_\nu h = e^{i\theta} h$. Then, $|h| = |e^{i\theta} h| = |\cL_\nu h| \leq \cL_F|h|$. But $\int|h|\dd\frm = \int \cL_F |h|\dd \frm$. So, it must be the case that $\cL_F|h| = |h|\,.$ Since $h \in \cB_s$ we have that $|h| \in \cB_s$. So, recalling \eqref{eq:projection}, $$|h| = \sum_j \frac{\int |h|\mathds{1}_{D_j}\dd\frm}{\mu_*(E_j)} \rho_j\,.$$ 
In particular, there is $j_0$ such that $|h|>0$ on $\supp \rho_{j_0}$. 
Write $h = |h|e^{ig_\nu}$ on $\supp \rho_{j_0}$. Since $h$ and $|h|$ are continuous, we can take $g_\nu$ to be continuous.  
Also, note that $\cL_\nu(h) = e^{i\theta}h = e^{\theta+g_\nu}|h|$ which implies that $\cL_F(e^{i(\nu \tau + g_\nu-g_\nu\circ F - \theta)}|h|) = |h| = \cL_F(|h|)\,.$ Then, $$\int_{\supp \rho_{j_0}} \im\, (e^{i(\nu \tau + g_\nu-g_\nu\circ F - \theta)}|h|)\dd\frm = \int_{\supp \rho_{j_0}} \im\, \cL_F(e^{i(\nu \tau + g_\nu-g_\nu\circ F - \theta)}|h|)\dd \frm = 0\,.$$ Therefore, 
$$\int_{\supp \rho_{j_0}} \sin (\nu \tau + g_\nu-g_\nu\circ F - \theta) |h| \dd \frm =0\,. $$
Hence, $\nu \tau+g_\nu-g_\nu\circ F+\theta=2\pi N$ for some $N:\supp \rho_{j_0}\to \bZ$. This contradicts assumption \textbf{(B2)} because it implies that
\[
\tau=\frac{2\pi N}{\nu}-\frac{\theta}{\nu},
\]
up to a continuous coboundary term.
\end{proof}

\section{Proofs of the main results}\label{sec:proof-thmA} 

\begin{proof}[Proof of \Cref{thmA}]

We verify the conditions in \Cref{app:AbstractCLT} as follows. Take $\bT^2$ with Borel $\sigma$-algebra and Lebesgue measure, $\frm$, as $(E, \mathcal{E}, \frm)$, the Banach algebra introduced in \Cref{sec:TraOp}, $\cB_s$, as $\cB$, and an SVPH, $F$, as the non-singular dynamical system.  Then, conditions (1),(2) and (3) follows from Proposition \ref{prop:spectral}. 

Both conditions (4) and (5-r) for all $r$ follow from \cite[Lemma 4.6]{FePe}. In fact, $\nu \mapsto \cL_\nu(\cdot) = \cL_F( e^{i\nu\tau } \cdot)$ is analytic. To see this, note that 
$$\cL_F( e^{i\nu\tau } \cdot) = \sum_{k=0}^{\infty} \frac{(i\nu)^k}{k!} \cL_F(\tau^k \cdot) $$
and the series converges in $\cB_s$ for all $\nu$ due to Weierstrass M-test because $$\|\cL_F(\tau^k \cdot)\| \leq \|\tau\|^k_{C^{r-1}}\|\cL_F\|\,.$$
\end{proof}


\begin{proof}[Proof of \Cref{thmB}]
The first part of the theorem, the LLT \eqref{fLLT}, follows directly from \Cref{thm:Hehe3}
because we have already verified the conditions in the theorem in the previous proof, and we make the assumption \textbf{(B2)} of $\tau$ being non-arithmetic. 

For the second part, let $\dd m = f_m\, \dd\frm$ be the initial probability measure.  Note that using the duality of the transfer operator, the characteristic function of $\tau_n = \sum_{k=0}^{n-1}\tau \circ F^k$ can be coded as $$\mathbb E_{m}(e^{i\nu \tau_n})= \frm(\cL^n_\nu(f_m))\,.$$
Due to \Cref{eq:EigenSplit}, there exists $\varepsilon>0$ such that for all $\nu \in (-\varepsilon,\varepsilon)$,
$$\mathbb E_{m}(e^{i\nu \tau_n}) = \sum_{k=1}^\ell\lambda_k(\nu)^n  \frm(\rho_k)\mu_k(f_m) +  \frm( K_n(\nu)f_m) +\frm(\Lambda^n(\nu)f_m)$$
where all the symbols are as in \Cref{thm:Hehe1}. 

As in the proof of \cite[Theorem A']{HeHe}, we have that
$$\lambda_k\left(\frac{\nu}{\sqrt{n}}\right)^n =\sum_{k=1}^\ell \frm(\rho_k)\mu_k(f_m) e^{-\frac{\sigma^2_k\nu^2}{2}+O(\frac{1}{\sqrt{n}})}\,.$$
Also, we know that $|\frm( K_n(\nu)f_m)| \leq C\max_k{|\lambda_k(\nu)|^n} \inf\{|\nu|,1\} \|\frm\|\|f\|\,.$

Let $L \in \mathbb R$ be such that $0<\varepsilon<L$. From \Cref{lem:NoUnitEigen}, it follows that there exists $c\in (0,1)$ such that, for all $\nu$ such that $|\nu| \in [\varepsilon, L]$,  $|\mathbb E_{m}(e^{i\nu \tau_n})|\leq Cc^n\,.$

Let $Z$ be an independent and a zero average random variable such that $|Z|\leq 1$ and with a smooth density $\psi$. Consider the random variable $\tilde{\tau}_n = \frac{\tau_n}{\sqrt{n}}+rZ$. Then $\tilde{\tau}_n$ has a density, $\mathcal{N}_{n, r}$, that satisfies

\[
\begin{split}
\mathcal{N}_{n,r}(y)= & \frac{1}{2 \pi} \int_{\mathbb{R}} e^{-i\nu y} \mathbb{E}\left(e^{i\nu \tilde{\tau}_n}\right) d\nu \\
= & \frac{1}{2 \pi} \int_{\mathbb{R}} e^{-i\nu y} \mathbb E_{m}(e^{i\frac{\nu}{\sqrt{n}} \tau_n}) \widehat{\psi}(\varepsilon\nu) d\nu \\
= & \frac{1}{2 \pi} \int_{-\varepsilon\sqrt{n}}^{\varepsilon\sqrt{n}} e^{-i\nu y}\left[\sum_{k=1}^\ell c_k e^{-\frac{\sigma^2_k\nu^2}{2}+O(\frac{1}{\sqrt{n}})}+\frm( K_n\left(\frac{\nu}{\sqrt{n}}\right)f_m)+O\left(c^n\right)\right] \widehat{\psi}(\varepsilon\nu) d\nu \\
&\quad\quad+O\left(c^n\right)+\frac{1}{2 \pi} \int_{|\nu| \geqslant L\sqrt{n} } e^{-i\nu y} \mathbb E_{m}(e^{i\frac{\nu}{\sqrt{n}} \tau_n})  \widehat{\psi}(\varepsilon\nu) d\nu
\end{split}
\]
where $\, \widehat{\cdot}\,$ denotes the Fourier transform. To obtain the complete asymtptotics, we use \Cref{thm:Hehe1} and note that  for $n$ sufficiently large, there exists $c>0$ such that
\begin{align*}
    \int_{- \varepsilon\sqrt{n}}^{ \varepsilon\sqrt{n}}\frm(K_n\left(\frac{\nu}{\sqrt{n}}\right)f_m)\, \dd\nu &\lesssim \int_{-1}^{1}\frac{t}{\sqrt{n}}\,\dd\nu+\int_{1\leq|\nu|\leq \varepsilon\sqrt{n}}\frac{|t|}{\sqrt{n}} e^{-ct^2/4}\,dt\\ &=O\left(\frac{1}{\sqrt{n}}\right)+O\left(\frac{1}{\sqrt{n}}\right)e^{-ct^2/4}\Big|_{1}^{ \varepsilon\sqrt{n}} = O\left(\frac{1}{\sqrt{n}}\right). 
\end{align*}

Using that $|\widehat{\psi}(\nu) |\lesssim |v|^{-\delta}$ for all $\delta>0$, we have
$$\mathcal{N}_{n, r}(y) = \sum_{k=1}^\ell\frac{c_k}{\sqrt{2\pi \sigma^2_k}} e^{-\frac{y^2}{2\sigma^2_k}} + O \left(\frac{1}{\sqrt{n}}+\frac{1}{\ve^{\delta}n^{(\delta-1)/2}}\right)\,$$
where $c_k =\frm(\rho_k)\mu_k(f_m)$.

So, to have a meaningful estimate, we set $\delta>2$. From now on we can argue almost exactly as in the proof \cite[Theorem 1.22]{DeKiLi}: on the one hand we have,
\[
\begin{split}
\mathbb{P}_m\left(\frac{\tau_n}{\sqrt n} \in[a, b]\right)&\le \mathbb{P}_m\left(\frac{\bar{\tau}_n}{\sqrt n} \in[a-\ve, b+\ve]\right)\\
& \leqslant \int_{a-\varepsilon}^{b+\varepsilon}\sum_{k=1}^\ell\frac{c_k}{\sqrt{2\pi \sigma^2_k}} e^{-\frac{y^2}{2\sigma^2_k}}dy + (b-a)\,O \left(\frac{1}{\sqrt{n}}+\frac{1}{\ve^{\delta}n^{(\delta-1)/2}}\right)  \\
& \leqslant \sum_{k=1}^dc_k \int_{a}^b \fn\left(\frac{y }{\sigma_k}\right)\, dy+O\left(\varepsilon {\max_k \fn\left(\frac{a}{\sigma_k}\right)}\right)\\
&\quad\quad + (b-a)\,O\left(\frac{1}{\sqrt{n}}+\frac{1}{\ve^{\delta}n^{(\delta-1)/2}}\right)
\end{split}
\]
and similarly for the lower bound. Choosing $\ve=n^{\frac{1}{\delta}-\frac 12}$ we conclude that the error estimate is up to a constant times
$$\left(\frac{\max_k \fn\big(\frac{a}{\sigma_k}\big)}{n^{\frac 12-\frac 1\delta}}+\frac{b-a}{\sqrt{n}}\right)\,.$$ 
The statement on the $\sigma_k$, and the conclusion, follow from the Green-Kubo formula
\[\sigma^2_k = \mu_k(\tau_k^2)+2\sum_{j=1}^\infty \mu_k(\tau_k\cdot\tau_k\circ F^j).\qedhere\]
\end{proof}

\begin{appendix}

\section{Abstract functional analytic framework for the CLT}\label{app:AbstractCLT}

Let $(E, \mathcal{E}, \frm)$ be a measure space. Suppose $(\cB, \|\cdot\|)$ is a complex Banach lattice\footnote{$f\in \cB \implies |f|, \bar f \in \cB$} and algebra\footnote{for all $f, g \in \cB$, $f\cdot g \in \cB$} of measurable function on $E$. 

Let $F$ be a non-singular transformation on $E$. Let the Transfer operator $\cL_F : \cB \to \cB$ be defined by 
$$\frm (f \circ F \,\, g) = \frm (f \,\cL_F(g) ),\,\,\,\forall g \in L^{\infty}\,.$$
Suppose that:
\begin{enumerate}
    \item  $\cL_F$ is quasi-compact with spectral radius $r(\cL_F)=1$,\footnote{$\cB = J \oplus H$ with $J,H$ $\cL_F$-invariant $r(\cL_F|_H) < 1$, $\dim J < \infty$, and each eigenvalue of $\cL|_J$ has modulus $1$.} and $\sup_n \|\cL^n\| < \infty\,.$
    \item There exist $\ell$ non-negative functions of $\cB$, denoted by $\rho_1,\dots,\rho_\ell$ that forms a basis for Ker$(\cL_F - \Id) \cap \cB\,$ and $\frm\left(\bT^2 \setminus \bigcup_{j=1}^d \supp \rho_j\right) =0\,.$\footnote{This assumption can be dropped by restricting the spaces to $\bigcup_{j=1}^\ell \supp \rho_j$ as explained in \cite{HeHe}.}
    \item There exist $\cL_F-$invariant probability measures $\mu_1, \dots,\mu_\ell$ in $\cB^\prime$ such that $\mu_k(\rho_m) = \delta_{k,m}$. 
\end{enumerate}

\vspace{10pt}

Let $\tau$ be a measurable function and $\nu \in \bR$ and the twisted transfer operators $\cL_\nu$ be
$$\cL_{\nu}(\cdot) := \cL_F(e^{i\nu \tau} \cdot)\,.$$
Suppose that:
\begin{enumerate}[start=4]
    \item  For all $\nu$, $\cL_\nu$ is a bounded linear operator on $\cB$.
    \item[(5-r)]There exists $\delta>0$ such that $\nu \mapsto \cL_{\nu}$ is $r$ times continuous differentiable as a function from $(-\delta, \delta)$ to $\fC(\cB,\cB)$, the space of linear bounded operators on $\cB$.\footnote{The derivative can be formally computed, e.g., $k$th derivative at $\nu$ is $\cL((i\tau)^ke^{i\nu\tau}\cdot)$} 
\end{enumerate}

\begin{oss}
In \cite{HeHe}, the assumptions cover the more general case of $\bigcup_{j=1}^d \supp \rho_j$ not having full measure and construct an auxiliary Banach space of functions restricted to $\bigcup_{j=1}^d \supp \rho_j$ (see \cite[Chapter XI]{HeHe}). However, for our applications this is too general. So, we make the assumption that $\bigcup_{j=1}^d \supp \rho_j$ is a full $\frm-$measure set, and we no longer require this auxiliary Banach space construction. Furthermore, due to the discussion in \cite{AV}, the assumption that $\delta_x \in \cB^\prime$ in \cite{HeHe} can also be dropped. 
\end{oss}

\begin{oss}
Define spaces $\cB_j, j=1,\dots,d$ as follows
$$\cB_j = \{ f|_{\supp \rho_j} | f \in \cB\}$$
and equip them with the norm $\|f\| = \inf \left\{\|g\| \big|\,g\in \cB\,, g|_{\supp \rho_j}=f \right\}$\,. Then $B_j$'s are Banach lattices (see \cite{HeHe}), and because $\cB$ is an algebra, $\cB_j$'s are algebras too. 
\end{oss}

Now, we can state the abstract results in \cite{HeHe} we have used in our paper which are adapted version of \cite[Theorem 2.3]{HeHe} and \cite[Theorem 4.1]{HeHe}.

\begin{thm}\label{thm:Hehe1}
Let $\tau \in L^r(\rho_j)$ for all $j$. Suppose the Assumptions (1) to (5-r) hold. Then there exist $C>0$, $c \in (0,1)$, $\varepsilon>0$, $C^r$ maps $\nu \mapsto \lambda_{k}(\nu), \nu \mapsto K_n(\nu), \nu \mapsto \Lambda(\nu)$ such that $\lambda_k(0)=1$, $\lambda^{\prime}_k(0)= i \mu_k(\tau)$, $\lambda^{\prime\prime}_k =: \sigma^2_k \geq 0$ and such that for all $n \geq 1$,  for all $\nu \in (-\varepsilon,\varepsilon)$, for all $f \in \cB$ and $g \in \cB^{\prime}$,
\begin{equation}\label{eq:EigenSplit}
    \langle  g , \cL_\nu (f) \rangle = \sum_{k=1}^\ell\lambda_k(\nu)^n  \langle g ,\rho_k \rangle \langle \mu_k , f \rangle + \langle g , K_n(\nu)f \rangle+\langle g , \Lambda^n(\nu)f \rangle\,.
\end{equation}
In addition, the following hold. 
\begin{enumerate}
    \item $|\langle g , K_n(\nu)f \rangle| \leq C\max_k{|\lambda_k(\nu)|^n} \inf\{|t|,1\} \|g\|\|f\|$\,.
    \item $|\langle g , \Lambda^n(\nu)f \rangle| \leq C c^n \inf\{|t|,1\} \|g\|\|f\|$\,.
\end{enumerate}
\end{thm}

\begin{thm}\label{thm:Hehe2} Let $\tau \in L^r(\rho_j)$ for all $j$ and assume there do not exist $c$ and $\xi \in  L^r(\rho_j)$ such that $\tau|_{\supp \rho_j} = c + \xi\circ F - \xi$. Suppose the Assumptions (1) to (5-r) hold
\begin{itemize}
    \item with $r=2$, then the general CLT \eqref{CLT} holds.  
    \item with $r=3$, then the general Berry-Esseen \eqref{BEst} estimate holds.
\end{itemize}
Further assume that for all $j$ there do not exist $a>0$ and $b \in \bR$ and $\xi \in  L^r(\rho_j)$ such that $\tau|_{\supp \rho_j} - \xi\circ F + \xi$ is $a\bZ+b$ valued. Suppose the Assumptions (1) to (5-r) hold with $r=2$, then the general LLT \eqref{fLLT} holds. 
\end{thm}

\begin{thm}\label{thm:Hehe3} Let $\tau \in L^r(\rho_j)$  and assume that for all $j$ there do not exist $a>0$ and $b \in \bR$ and $\xi \in  L^r(\rho_j)$ such that $\tau|_{\supp \rho_j} - \xi\circ F + \xi$ is $a\bZ+b$ valued. Suppose the Assumptions (1) to (5-r) hold with $r=2$, then the general LLT \eqref{fLLT} holds.
\end{thm}

\section{Differential operators}
In this section we proof a Lemma about the commutation of the twisted transfer operator and the differential operators.

Before giving the statement, let us recall some notation and formulas derived in \cite{CaLi} for the commutator of the transfer operator $\cL_F$ with differential operators. All the derivatives in the following are meant in the weak sense. For a smooth function $u:\bR^2\to \bR$, a diffeomorphism $F:\bR^2\to \bR^2$ and a vector $v \in \bR^2$, 
define
\[
\begin{split}
&\partial_v u:=\langle \nabla u, v \rangle,		\\
&\partial_{{F^*}v} u:=\langle \nabla u, (DF)^{-1}v \rangle
\end{split}
\]
viewed as functions from $\bR^2$ to $\bR$. In addition, for any multi-index $\alpha=(\alpha_1,..,\alpha_k)\in \{1,2\}^k$ with $k\le s$, we define the linear operator $\partial^\alpha$ acting on $\cC^\infty(\bT^2)$ as
\[
\partial^\alpha u=\partial_{x_{\alpha_k}} \cdots \partial_{x_{\alpha_1}} u.
\]

Then, by \cite[Corollary 6.3]{CaLi},
\begin{equation}\label{eq: LnP=PLn}
\partial^{\alpha} \cL_F^n u=\cL_F^n P^{\alpha}_{n, k} u
\end{equation}
where\footnote{We denote by $\{e_{i}\}_{i=1}^2$ the usual orthonormal base of $\bR^2$.}
\begin{equation}\label{definition of P}
\begin{split}
&P^{\alpha}_{n, 0}u=u,\\
& P^{\alpha}_{n, 1}u=A_n^{\alpha_1}u-A_n^{\alpha_1}J_n\cdot u,\\
&\quad \vdots 
\\
&P^{\alpha}_{n,k}u=A^{\alpha}_{n,1}u
-\sum_{j=1}^{k}A^{\alpha}_{n,j+1}((A^{\alpha_{j}}_n J_n)
\cdot  P^{\alpha}_{n, j-1}u) \quad \textit{ for } k\ge 2,
\end{split}
\end{equation}
with $A_n^{\alpha_j}=\partial_{F^{n^*}e_{\alpha_j}}$, $A_{n,j}^{\alpha}:=A_n^{\alpha_k}\cdots A_n^{\alpha_{j}}$, $A_{n, 
k+1}^{\alpha}=\text{Id}$ and $J_n:=\log |\det DF^n|$. For example, for  $s=1$, $\alpha_j$ corresponds to $\partial_{x_j}$ and we have
\begin{equation}\label{Lprime}
\begin{split}
\partial_{x_j} \mathcal{L}_F^n u= \mathcal{L}_F^n\left(\partial_{F^{n^*}e_j} u -\partial_{F^{n^*}e_j}  J_n\cdot  u\right)=\mathcal{L}_F^n \left( P^{\alpha}_{n, 1}u\right).
\end{split}
\end{equation}

Note that there exists $\Lambda>0$ such that, setting $d_{j,i}=\langle (DF^n)^{-1}e_{j}, e_j\rangle$, then $A_n^{\alpha_j} =\sum_{i=1}^2d_{j,i}\partial_{x_j}$ (i.e. $A_n^{\alpha_j} $ are differential operators) and   
\begin{equation}\label{diffop}
\|d_{j,i}\|_{\cC^k}\le \|(DF^{n})^{-1}\|_{\cC^k}\le \Lambda^{n},
\end{equation}
for each $1\le k\le s.$

\begin{lem}\label{lem:commute}
For each $s \ge 1, n\in \bN$ and for each $k\in \{1,..,s\}$, there exist\\ $m_k, \{a_{j,k}\}_{j\le m_k},\{ \beta_{j,k}\}_{j\le m_k}, \{\gamma_{j,k}\}_{j\le m_k}\in \bN$, with $j_k\le m_k$, and  operators $\{\cD^u_{j,k}\}_{j\le m_k}$ and $\{\cD^\tau_{j,k}\}_{j\le m_k}$  such that, for each $\nu$ and $\alpha$,
\begin{equation}\label{commute}
\partial^\alpha \cL^n_\nu u=\cL_F^n(e^{i\nu \tau_n}P^{\alpha}_{n,s}u)+\cL^n_F\left( e^{i\nu\tau_n}\sum_{k=1}^s (i\nu)^k\sum_{j=1}^{m_k}  a_{j, k}(\cD^u_{j,k}u)^{\beta_{j,k}}(\cD^\tau_{j,k}\tau_n)^{\gamma_{j,k}}\right).
\end{equation}
Moreover, for each $k\in \{1,..,s\}$ and $j\le m_k$, $\cD^u_{j,k}$ are differential operators of order smaller than $s-1$ and $\cD^{\tau}_{j,k}$ are differential operators of order smaller than $s$. Finally, the $\cC^k$ norm of the coefficients of $\cD^u_{j,k}$ and $\cD^{\tau}_{j,k}$ is bounded by $C_{s} \Lambda^n$, for some $C_s>0$.
\end{lem}
\begin{proof}
We are going to show the cases $s=1$ and $s=2$. The general case is just a straightforward generalization of those two using Leibniz formula for higher derivatives, and it is left to the reader. \\
For each vector $v\in \bR^2$
\[
\partial_{F^{n^*}v}(e^{i\nu\tau_n} u)=e^{i\nu\tau_n}(\partial_{F^{n^*}v} u+i\nu u \partial_{F^{n^*}v} \tau_n
),\]
and, for $\alpha=\alpha_1$,
\begin{equation}\label{eq:Palpha}
P^\alpha_{n,1}(e^{i\nu \tau_n} u)=e^{i\nu\tau_n}(P^\alpha_{n,1} u+i\nu A^{\alpha}_n \tau_n).
\end{equation}
Therefore, using $\cL^n_\nu(u)=\cL_F^n(e^{i\nu\tau_n}u)$ and \eqref{Lprime},  we have 
\[
\begin{split}
\partial_{x_1} \mathcal{L}_{\nu}^n u&= \mathcal{L}_F^n\left(e^{i\nu\tau_n}(\partial_{F^{n^*}e_1} u -\partial_{F^{n^*}e_1}  J_n\cdot  u)\right)+i\nu \cL_F^n(e^{i\nu\tau_n} u \cdot \partial_{F^{n^*}e_1} \tau_n )\\
&=\mathcal{L}_F^n\left(e^{i\nu\tau_n}P^{\alpha}_{n, 1}u\right)+i\nu \cL_F^n(e^{i\nu\tau_n} u \cdot A_n^{\alpha}\tau_n ).
\end{split}
\]
An analogous computation yields the formula for $x_2$. We can thus obtain \eqref{commute} with 
\[
 m_1=a_{1,1}=\beta_{1,1}=\gamma_{1,1}=1,\quad \cD^u_{1,1}=\Id,\quad\cD^\tau_{1,1}=A^{\alpha}_n.
\]
The bound on the norm of coefficients of $A^{\alpha}_n$ is given by \eqref{diffop}.\\
Differentiating again, for example in the $x_1$-direction, and using the above formula for the first derivative and \eqref{Lprime}, we have
\[
\begin{split}
\partial_{x_1}^2 \cL^n_\nu u&= \partial_{x_1}(\mathcal{L}_F^n\left(e^{i\nu\tau_n}P^{\alpha}_{n, 1}u\right))+i\nu \partial_{x_1}(\cL_F^n(e^{i\nu\tau_n} u \cdot A_n^{\alpha}\tau_n ))\\
&=\cL_F^n(P^{\alpha}_{n,1}(e^{i\nu \tau_n}P^{\alpha}_{n,1}u))+i\nu \cL_F^n(P^\alpha_{n,1}(e^{i\nu\tau_n} u \cdot A_n^{\alpha}\tau_n) ).
\end{split}
\]
Hence, using \eqref{eq:Palpha}, we obtain
\[
\begin{split}
\partial_{x_1}^2 \cL^n_\nu u&=\cL_F^n(e^{i\nu \tau_n}P^{\alpha}_{n,1}(P^{\alpha}_{n,1}u))\\
&+\cL^n_F\left( e^{i\nu\tau_n}(2i\nu (P^\alpha_{n,1}u)\cdot(A_{n,1}^\alpha \tau_n) +i\nu u\cdot P^{\alpha}_{n,1}(  A_{n,1}^\alpha \tau_n)-\nu^2 u\cdot ( A_{n,1}^\alpha \tau_n)^2 ) \right).
\end{split}
\]
A similar computation holds for the derivative in the $x_2$-direction.
So we obtain \eqref{commute} with $m_1=2, m_2=1$ and
\[
(a_{1,1}, a_{2,1}, a_{1,2})=(2,1,1), \quad (\beta_{1,1},\beta_{2,1}, \beta_{1,2})=(1,1,0), \quad (\gamma_{1,1},\gamma_{2,1}, \gamma_{1,2})=(1,1,2)
\]
and
\[
\begin{split}
&(\cD_{1,1}^u, \cD_{2,1}^u, \cD_{1,2}^u)=(P^\alpha_{n,1}, Id,Id)\\
&(\cD_{1,1}^\tau, \cD_{2,1}^\tau, \cD_{1,2}^\tau)=(A^\alpha_{n,1}, P_{n,1}^\alpha \circ A^\alpha_{n,1},A^\alpha_{n,1}).
\end{split}
\]
The $\cD_{j,k}^u$ operators, where $k\in\{1,2\}$, are up to order $1$, while the $\cD_{j,k}^u$ operators are up to order $2$. The bound on the norm of the coefficients is given again by \eqref{diffop}.
\end{proof}
\section{Complex Banach spaces}\label{app:complex}
We recall here a standard procedure to complexify a Banach space $(\cB,\|\cdot \|)$. See \cite[Section I]{Sin} for the details. We consider a real vector space $\cB^2$ and for $(x,y)\in \cB^2$ we define the multiplication by $a+ib \in \bC$ as $(a+ib)\star (x,y)=(ax-by, ay+bx)$. We can thus define the norms
\[
\begin{split}
&\|(x,y)\|_{\cB, \bR}=\sqrt{\|x\|^2+\|y\|^2}\\
&\|(x,y)\|_{\cB, \bC}=\sup_{\theta \in [0,2\pi]}\|e^{i\theta} \star (x,y)\|_{\cB,\bR}.
\end{split}
\] 
The space $(\cB^2, \|\cdot\|_{\cB,\bR})$ is a real Banach space and we can define the complex Banach space $(\bB, \|\cdot\|_{\cB,\bC})$ whose elements are $x+iy$, $x,y \in \cB$. The following easy inequalities hold:
\begin{equation}\label{eq: complex}
\begin{split}
&\|f\|_{\cB,\bC}\le 2\sqrt 2 \|f\|_{\cB,\bR}\\
&\|x\|+\|y\| \le \sqrt 2 \|(x,y)\|_{\cB,\bR}\le \sqrt 2 \|x+iy\|_{\cB,\bC}.
\end{split}
\end{equation}

\section{Previous results on  \texorpdfstring{$F$}{Lg} and \texorpdfstring{$\cL_F$}{Lg}  }
For the convenience of the reader, in this Appendix we collect some consequences of several results obtained in \cite{CaLi} for the SVPH system $F$ and the associated transfer operator $\cL=\cL_F$ which are used throughout the paper.\\

\begin{itemize}
\item There are constant cones $\fC_u \supset e_1$ (which is $DF$-invariant) and $\fC_c \supset e_2$ (which is $(DF)^{-1}$-invariant), numbers $\lambda> \lambda_c>1$ and $C_1,C_2>0$ such that 
\begin{equation}\label{PHS}
\begin{split}
&\|DF^n|_{\fC_u}\|> C_1\lambda^{n}\\
(C_2 \lambda_c^n)^{-1}<&\|DF^n|_{\fC_c}\|<C_2 \lambda_c^n
\end{split}
\end{equation}
\item By the second equation of \cite[(5.17)]{CaLi}:
\begin{equation}\label{DF-1}
   { \|(DF^n)^{-1}\|_{\cC^0(\bT^2)}}\le C\lambda_c^n.
\end{equation}
\item 
By \cite[Prop. 5.6]{CaLi} there exists a constant $C_*\ge1$ such that, for every $z\in \mathbb{T}^2$, any $n\in\bN$, any vectors $v^u\in \mathbf{C}_u$ and $v^c\in \mathbf{C}_c$ such that $(a,b):=D_zF^{n}v^c\not\in \mathbf{C}_u$, we have :
\begin{equation}\label{det-form}
 C_*^{-1}\frac{\|D_zF^n v^u\|}{\|v^u\|}\frac{ |b|}{\|v^c\|}\le |\det D_zF^n|\le C_* \frac{\|D_zF^n_z v^u\|}{\|v^u\|}\frac{ |b|}{\|v^c\|}.
\end{equation}
In particular: there is $C_*>0$ such that for each $n\in \bN$, $z\in \bT^2$ 
\begin{equation}\label{det-est}
C_*^{-1} \lambda_c^{-n} \lambda^{n}\le |\det DF^n(z)|\le C_* \Lambda^n \lambda_c^n
\end{equation}
where $\lambda=\inf_{(x,y)\in\bT^2} \partial_x f(x,y)$ and $\Lambda=\sup_{(x,y)\in\bT^2}\partial_x f(x,y)$.\\
\item By \cite[Lemma 5.18 and Corollary 5.20]{CaLi} 
\begin{equation}\label{C1norm}
\begin{split}
&\|\cL^n 1\|_{\cC^0(\bT^2)}\le C\lambda_c^n\\
&\|\cL^n 1\|_{\cC^1(\bT^2)}\le C\lambda_c^{2n}.
\end{split}
\end{equation}\\
\item By \cite[Corollary 6.7]{CaLi}, for each $q\in \bN$ there exists $C,C_\sharp$ and $\delta_*\in (0,1)$ such that $\delta_*\lambda_c<1$ and, for each  and $u\in \cC^{r-1}(\bT^2)$, $n\in \bN$,
\begin{equation}
\label{L-Y B2} 
\begin{split}
\|\cL_F^n u\|^*_0 &\le C\lambda_c^n \|u\|^*_0 \\
\|\mathcal{L}^n_F u \|^*_{q} &\le  C_\sharp\delta_\ast^{n} \|u\|^*_{q}+ C\lambda_c^n\|u\|^*_{0}.
\end{split}
\end{equation}
    \item By \cite[Lemma 7.1]{CaLi} and \eqref{C1norm}, $\cL_F$ is bounded as an operator in $H^s(\bT^2)$ and for each $n\in \bN$ there are $C>0$ such that,
\begin{align}
&\| \cL^n u\|_{L^2}\le C\lambda_c^n \|u\|_{L^2}\label{L2 norm transfer}\\
&\|\mathcal{L}^n u\|^2_{H^s}  \le  C\lambda_c^{2(s+1)n}
\|u\|^2_{H^s}+C \lambda_c^{3n}\|u\|^2_{\cH^{s-1}}\label{LY Hs and Hs-1}.
\end{align}
Moreover, iterating \cite[(9.9)]{CaLi} there exist $C>0$ and $\hat \sigma_s\in (0,1)$ which can be chosen (see \cite[(9.3)]{CaLi} and the assumptions on the map $F$) equal to\footnote{Recall that $\xi_s=6(s+1)!$.} $\lambda_c^{\xi_s}\lambda_-^{-1}$, such that for each $n,s\in \bN$
\begin{equation}\label{LY-HB0}
\|\cL^n u\|_{H^s}\le \hat \sigma^n\|u\|_{H^s}+C\|u\|^*_{s+2}.
\end{equation}
\item By \cite[Lemma F.1, (G2) ]{CaLi} there exists\footnote{In our case the quantity $C_{\lambda_c,n}$ which appear in \cite{CaLi} can be bounded by a constant uniform in $n$ which might depends on $\lambda_c$ (see \cite[Remark 5.7]{CaLi}).} $C>1$ such that, for each $z_1,z_2\in \bT^2$, $v\in \fC^u$ unitary and $n\in \bN$
\begin{equation}\label{Lipsh}
\|D_{z_1}F^{n}v-D_{ z_2}F^{n}v\|\le L_{n}\|z_1-{z}_2\|
\end{equation}
where $L_n=C\lambda_c^{5n}$.
\end{itemize}

\end{appendix}


\begin{thebibliography}{99}
\bibitem{AV} Aimino, R; Vaienti, S. \textit{A Note on the Large Deviations for Piecewise Expanding Multidimensional Maps}. In: González-Aguilar, H., Ugalde, E. (eds) Nonlinear Dynamics New Directions. Nonlinear Systems and Complexity, vol 11. (2015) Springer.

\bibitem{ABV}  Alves, J; Bonatti, C.; Viana, M. \textit{SRB measures for partially hyperbolic systems whose central direction is mostly expanding.} Invent. Math. {\bf 140} (2000), no. 2, 351--398.

\bibitem{AlGe} S. Alinhac, P.Gérard, {\it Pseudo-differential Operators and the Nash-Moser Theorem}, Graduate Studies in Mathematics Volume: 82; 2007; 168 pp 

\bibitem{AGT} Avila, A; Gou\"ezel, S; Tsujii, M. \textit{Smoothness of solenoidal attractors}, Discrete and Continuous Dynamical Systems 15 (2006), no. 1, 21-35.

\bibitem{BaCa}  Viviane Baladi, Roberto Castorrini. {\it Thermodynamic formalism for piecewise expanding maps in finite dimension.} Discrete and Continuous Dynamical Systems. doi: 10.3934/dcds.2024023. 

\bibitem{BaGo} V. Baladi, S. Gou\"ezel {\it Good Banach spaces for piecewise hyperbolic maps via interpolation}  Annales de l'Institut Henri Poincaré / Analyse non linéaire 26 (2009) 1453-1481 


\bibitem{Bes} Besicovitch, A. S. (1945), \textit{A general form of the covering principle and relative differentiation of additive functions, I}, Proceedings of the Cambridge Philosophical Society, 41 (02): 103–110.

\bibitem{BuCaCa} Oliver Butterley, Giovanni Canestrari, Roberto Castorrini \textit{Discontinuities cause essential spectrum on surfaces},  Ann. Henri Poincaré (2024). \url{https://doi.org/10.1007/s00023-024-01499-y}.

\bibitem{BV} Bonatti C.; Viana, M.
 \textit{SRB measures for partially hyperbolic systems whose central direction is mostly contracting.}
Israel J. Math. {\bf 115} (2000), 157--193.
 
 \bibitem{BuWilk} Burns, Keith; Wilkinson, Amie.
\textit{On the ergodicity of partially hyperbolic systems.}
Ann. of Math. (2) 171 (2010), no. 1, 451--489. 

\bibitem{BuEs} Butterley, O.; Eslami, P. \textit{Exponential Mixing for Skew Products with Discontinuities.} Amer. Math. Soc., 369:783803, (2017).


\bibitem{CaLi} Castorrini, R.; Liverani, C. {\it Quantitative statistical properties of two-dimensional partially hyperbolic systems}, Advances in Mathematics, 409, Part A, 1-122 (2022). 


\bibitem{DeKiLi}  M.~F.~Demers, N.~Kiamari, C.~Liverani, Transfer operators in hyperbolic dynamics. An introduction. $33^o$  Col\'oq. Bras. Mat. Instituto Nacional de Matem\'atica Pura e Aplicada (IMPA), Rio de Janeiro, 2021.


\bibitem{DeLi2} De Simoi, J.; Liverani, C.~\textit{Limit theorems for fast-slow partially hyperbolic systems.} Invent. Math. 213 (2018), no. 3, 811--1016.



\bibitem{F11} F. Faure.  \textit{Semiclassical origin of the spectral gap for transfer operators of a partially expanding map.} Nonlinearity {\bf 24} (2011), no. 5, 1473--1498.

\bibitem{Dima98} D.~Dolgopyat.
\newblock  \textit{On decay of correlations in {A}nosov flows.}
\newblock { Ann. of Math.}, 147(2):357--390, 1998.

\bibitem{Dol1} D.~Dolgopyat.
\newblock  \textit{Limit theorems for partially hyperbolic systems.}
\newblock { Trans. Amer. Math. Soc.}, 356, 1637-1689, 2004.

 \bibitem{FeLi} Fernando, K.; Liverani C.~{\it Edgeworth expansions for weakly dependent random variables}. Ann.~Inst.~Henri Poincaré Probab.~Statist., volume 57, number 1 (2021) 469--505.
 
  \bibitem{FePe} Fernando, K.; P{\`e}ne, F.~{\it Expansions in the local and the central limit theorems for dynamical systems}. Communications in Mathematical Physics, volume 389, number 1 (2022) 273--347.



\bibitem{Gou1} S. Gou\"ezel, \textit{Local limit theorem for nonuniformly partially hyperbolic skew-products and Farey sequences}, Duke Math.~J.~{\bf 147} (2009), no.~2, 193--284.

\bibitem{Hen} H. Hennion. \textit{Sur un th\'eor\`eme spectral et son application aux noyaux lipchitziens.} Proc.~Amer.~Math.~Soc.~118:2 (1993), 627--634. 


\bibitem{HeHe} Hennion, H.; Herv\'e, L. {\it Limit Theorems for {M}arkov Chains \& Stochastic Properties of Dynamical Systems by Quasi-Compactness}, 2001, Springer-Verlag



\bibitem{KKM} A.~Korepanov, Z.~Kosloff and I.~Melbourne, Averaging and rates of averaging for uniform families of deterministic fast-slow skew product systems, Studia Math.~{\bf 238} (2017), no.~1, 59--89.

\bibitem{SM} I. Melbourne and A. M. Stuart. A note on diffusion limits of chaotic skew-product flows.
Nonlinearity, 24(4):1361–1367, 2011.

\bibitem{Liv} Liverani C., {\it Transport in partially hyperbolic fast-slow systems.} Proceedings of the International Congress of Mathematicians--Rio de Janeiro 2018. Vol. III. Invited lectures, 2629-2654 Talk to the ICM 2018.


\bibitem{Sin} Singer, Ivan, Bases in Banach Spaces I, Springer (1970).


\bibitem{Tsu} Tsujii, M. \textit{Physical measures for partially hyperbolic surface endomorphisms}, Acta Math.
194 (2005), 37-132.





\end{thebibliography}
\end{document}